\documentclass[reqno,12pt]{amsart}

\usepackage[utf8]{inputenc}
\usepackage{microtype}
\usepackage{amsfonts}
\usepackage{amsmath}
\usepackage{graphicx}
\usepackage{float}
\usepackage{color}
\usepackage{hyperref}
\usepackage{tikz-cd}

\usepackage{amssymb}
\usepackage{enumitem}

\usepackage{tikz}
\usetikzlibrary{topaths}
\usetikzlibrary{calc}

\usepackage{a4wide}

\usepackage{empheq}

\newtheorem{theorem}{Theorem}[section]
\newtheorem*{theorem*}{Theorem}

\newtheorem{lemma}[theorem]{Lemma}

\newtheorem{proposition}[theorem]{Proposition}
\newtheorem*{proposition*}{Proposition}
\newtheorem{corollary}[theorem]{Corollary}
\newtheorem*{corollary*}{Corollary}
\newtheorem{conjecture}[theorem]{Conjecture}

\newtheorem*{conjecture*}{Conjecture}

\newtheorem*{question*}{Question}

\newtheorem*{main:main}{Theorem~\ref{thrm:main}}
\newtheorem*{main:main_cor}{Corollary~\ref{cor:main_cor}}
\newtheorem*{main:main_circle}{Theorem~\ref{thrm:circle}}

\theoremstyle{definition}
\newtheorem{definition}[theorem]{Definition}
\newtheorem{remark}[theorem]{Remark}
\newtheorem{example}[theorem]{Example}

\newcommand{\Z}{\mathbb{Z}}
\newcommand{\N}{\mathbb{N}}
\newcommand{\Q}{\mathbb{Q}}
\newcommand{\R}{\mathbb{R}}

\newcommand{\lk}{\operatorname{lk}}
\newcommand{\st}{\operatorname{st}}
\newcommand{\alk}{\lk^\uparrow\!}
\newcommand{\asst}{\st^\uparrow\!}

\newcommand{\defeq}{\mathbin{\vcentcolon =}}

\DeclareMathOperator{\Hom}{Hom}

\DeclareMathOperator{\Homeo}{Homeo}

\DeclareMathOperator{\PSL}{PSL}
\DeclareMathOperator{\F}{F}

\DeclareMathOperator{\Stab}{Stab}
\DeclareMathOperator{\RStab}{RStab}

\newcommand{\Fbr}%
   {F_{\operatorname{br}}}                 

\newcommand{\Vbr}%
   {V_{\operatorname{br}}}                 
   
\numberwithin{equation}{section}

\begin{document}

\title[The BNSR-invariants of the Lodha--Moore groups]{The BNSR-invariants of the Lodha--Moore groups,
 and an exotic simple group of type $\F_\infty$}
\date{\today}


\author[Y.~Lodha]{Yash Lodha}
\address{Institute of Mathematics, EPFL, Lausanne, Switzerland}
\email{yash.lodha@epfl.ch}

\author[M.~C.~B.~Zaremsky]{Matthew C.~B.~Zaremsky}
\address{Department of Mathematics and Statistics, University at Albany (SUNY), Albany, NY 12222}
\email{mzaremsky@albany.edu}

\begin{abstract}
In this paper we give a complete description of the Bieri--Neumann--Strebel--Renz invariants of the Lodha--Moore groups. The second author previously computed the first two invariants, and here we show that all the higher invariants coincide with the second one, which finishes the complete computation. As a consequence, we present a complete picture of the finiteness properties of normal subgroups of the first Lodha--Moore group. In particular, we show that every finitely presented normal subgroup of the group is of type $\F_\infty$, answering question $112$ from Oberwolfach Rep., 15(2):1579-1633, 2018. The proof involves applying a variation of Bestvina--Brady discrete Morse theory to the so called cluster complex $X$ introduced by the first author. As an application, we also demonstrate that a certain simple group $S$ previously constructed by the first author is of type $\F_\infty$. This provides the first example of a type $\F_\infty$ simple group that acts faithfully on the circle by homeomorphisms, but does not admit any nontrivial action by $C^1$-diffeomorphisms, nor by piecewise linear homeomorphisms, on any $1$-manifold.
\end{abstract}

\maketitle
\thispagestyle{empty}

\section*{Introduction}
In \cite{lodha16}, the first author and Justin Moore constructed a finitely presented non-amenable group with no non-abelian free subgroups, which has come to be known as the Lodha--Moore group (the second author of the current paper has insisted to the first author that we retain this terminology here). The first author subsequently showed the group to be of type $\F_\infty$ \cite{lodha14}, making it the first known example of a non-amenable type $\F_\infty$ group with no non-abelian free subgroups. Recall that a group is of \emph{type $F_n$} if it admits a classifying space with finite $n$-skeleton, and of \emph{type $F_\infty$} if it is of type $\F_n$ for all $n$.

In fact, the group just mentioned is one of two closely related groups introduced in \cite{lodha16}, one contained in the other, and there are also two naturally occurring intermediate groups, so collectively we will refer to these four groups as the \emph{Lodha--Moore groups}. In \cite{zaremsky16} the Lodha--Moore groups are denoted $G$, $_yG$, $G_y$, and $_yG_y$. In \cite{lodha16} the first is denoted $G_0$ and the last is denoted $G$, and in \cite{lodha14} only the first is considered and denoted $G$. Following the proof in \cite{lodha14} that $G$ is of type $\F_\infty$, it is easy to see that all four Lodha--Moore groups are of type $\F_\infty$.

The Bieri--Neumann--Strebel--Renz (BNSR) invariants $\Sigma^n(\Gamma)$ ($n\in\N$) of a group $\Gamma$ of type $\F_\infty$ are a family of geometric objects encoding a great deal of information about $\Gamma$, which have historically proven to be quite difficult to compute. In \cite{zaremsky16}, the second author computed $\Sigma^1$ and $\Sigma^2$ for each Lodha--Moore group (the computations will be recalled in Section~\ref{sec:bnsr}).

The main result of the current paper is:

\begin{main:main}
For any Lodha--Moore group $H$ and any $n\ge 2$, we have $\Sigma^n(H)=\Sigma^2(H)$.
\end{main:main}

This then provides a complete picture of all the BNSR-invariants of all the Lodha--Moore groups. The Lodha--Moore groups are close relatives of Thompson's group $F$, and this computation continues the trend of Thompson-like groups $H$ satisfying $\Sigma^n(H)=\Sigma^2(H)$ for all $n\ge 2$. This holds for example for $F$ itself \cite{bieri10}, for the generalized Thompson groups $F_{n,r}$ \cite{zaremsky17F_n}, and for the braided Thompson group $\Fbr$ \cite{zaremsky18}.

This theorem also allows us to present a complete picture of the finiteness properties of normal subgroups of the first Lodha--Moore group $G$; see Subsection~\ref{sec:classify_normal}. Since every non-trivial normal subgroup of $G$ contains the commutator subgroup $[G,G]$ \cite{burillo18}, an easy consequence of our main result is the following:

\begin{main:main_cor}
Every finitely presented normal subgroup of $G$ is of type $\F_\infty$.
\end{main:main_cor}

The non-trivial normal subgroups of $G$ have a natural bijective correspondence with the subgroups of $\Z^3$, and thanks to the computation of the $\Sigma^n(G)$ this makes it easy to precisely characterize the finiteness properties of every normal subgroup of $G$ (see Corollary~\ref{cor:classify}). For the other Lodha--Moore groups $_yG$, $G_y$, and $_yG_y$ things are not as nice as for $G$, since not every non-trivial normal subgroup contains the commutator subgroup, but it still follows from our results that every finitely presented normal subgroup containing the commutator subgroup is of type $\F_\infty$. This also answers in the affirmative Question~112 from \cite{MFO}.

The strategy involves a certain contractible CW complex $X$, constructed by the first author in \cite{lodha14} to prove that $G$ is of type $\F_\infty$. The cells of $X$ are clumped together into so called clusters, which are certain subdivisions of cubes. The key for our purposes here is that subcomplexes of vertex links in $X$ tend to be contractible, and this allows us to use Bestvina--Brady discrete Morse theory to deduce our main result.

We also use our computation of the BNSR-invariants of the Lodha--Moore groups to prove that a certain infinite simple group $S$, constructed by the first author in \cite{lodha19} and closely related to the Lodha--Moore groups, is of type $\F_\infty$. Recall that Thompson's group $T$ is the group of piecewise $\PSL_2(\Z)$-projective homeomorphisms of the circle $\mathbb{S}^1=\R\cup \{\infty\}$ with breakpoints in the set $\Q\cup \{\infty\}$. The group $S$ is generated by $T$ together with the following piecewise projective homeomorphism of the circle:
\[{\bf s}(t)=
\begin{cases}
t&\text{ if }t\le 0\\
 \frac{2t}{1+t}&\text{ if }0\le t\le 1\\
 \frac{2}{3-t}&\text{ if }1\le t\le 2\\
t&\text{ if }t\ge 2
\end{cases}
\]

\begin{main:main_circle}
The group $S$ is of type $\F_\infty$.
\end{main:main_circle}

It was demonstrated in \cite{lodha14} that $S$ is finitely presented and simple, and an explicit presentation for $S$ was computed. In the same article, it was also demonstrated that the group does not admit actions of certain types. Combining this with the above, we obtain the following.

\begin{corollary*}
There exists a type $\F_\infty$ simple group that acts faithfully by homeomorphisms on the circle, but does not admit any non-trivial action by $C^1$-diffeomorphisms, nor by piecewise linear homeomorphisms, on any $1$-manifold.
\end{corollary*}

Groups of homeomorphisms of the circle provide a rich source of examples of simple groups of type $\F_\infty$. Historically, until the appearance of $S$ in the literature,  all such examples have been groups of piecewise linear homeomorphisms. These were Thompson's group $T$ and its various generalizations due to Higman (see \cite{higman74}) and Stein (see \cite{stein92}). A novel new example appears in \cite{belk20}. Moreover, in \cite{ghys87}, Ghys and Sergiescu proved that $T$ admits a faithful action by $C^\infty$-diffeomorphisms of the circle. In fact, they show that the standard actions of $T$ (indeed all minimal actions) are topologically conjugate to such actions. It remains an open question whether the groups of Higman and Stein admit faithful actions by $C^1$-diffeomorphisms of the circle, although it was shown in \cite{bonatti19} that certain Stein groups have no faithful $C^2$-action on the circle.

The group $S$ has additional dynamical and algebraic features that distinguish it from the piecewise linear examples. For instance, in \cite{lodha19} it was observed that the natural action of $S$ on $\mathbb{S}^1$ produces a non-amenable orbit equivalence relation. Also, $S$ contains subgroups isomorphic to the Baumslag--Solitar group $\textup{BS}(1,2)$, hence it contains distorted cyclic subgroups, in contrast to $T$, which has no distorted cyclic subgroups \cite[Theorem~1.9]{burillo18a}. This combination of features makes the group $S$ a rather special member of the family of type $\F_\infty$ simple subgroups of $\Homeo^+(\mathbb{S}^1)$.

One could hope to adapt the cluster complex construction for $G$ from \cite{lodha14} to the group $S$ to prove that it is of type $\F_\infty$, but this poses considerable technical challenges. The novelty of our approach is to instead apply our computation of the BNSR-invariants of $_yG_y$ to provide a succinct resolution to this problem.

This paper is organized as follows. In Section~\ref{sec:group} we recall the Lodha--Moore groups. In Section~\ref{sec:cluster_cpxes} we recall the notion of a cluster complex, and in Section~\ref{sec:our_complex} we discuss the specific cluster complex $X$. In Section~\ref{sec:bnsr_morse} we set up the relevant background on BNSR-invariants, and a version of Bestvina--Brady discrete Morse theory that is useful for our purposes. In Section~\ref{sec:main} we apply Morse theory to $X$ to finish the computation of the BNSR-invariants of the Lodha--Moore groups. Finally, in Section~\ref{sec:circle} we prove that the simple group $S$ from \cite{lodha19} is of type $\F_\infty$.

\subsection*{Acknowledgments} We are grateful to Matt Brin and Sang-hyun Kim for helpful comments on a first draft of this paper. The first author is supported by a Swiss national science foundation \emph{Ambizione} grant. The second author is supported by grant \#635763 from the Simons Foundation.

\section{The Lodha--Moore groups}\label{sec:group}

The Lodha--Moore groups are certain groups of self-homeomorphisms of the Cantor set $\{0,1\}^\N$, which we recall in this section. First let $x,y\colon \{0,1\}^\N\to \{0,1\}^\N$ be the homeomorphisms defined recursively by
\[
\xi \cdot x \defeq \left\{\begin{array}{ll}
0\eta & \text{if }\xi=00\eta\\
10\eta & \text{if }\xi=01\eta\\
11\eta & \text{if }\xi=1\eta
\end{array}\right. \text{ and }
\xi \cdot y \defeq \left\{\begin{array}{ll}
0(\eta\cdot y) & \text{if }\xi=00\eta\\
10(\eta\cdot y^{-1}) & \text{if }\xi=01\eta\\
11(\eta\cdot y) & \text{if }\xi=1\eta \text{.}
\end{array}\right.
\]
Let $\{0,1\}^{<\N}$ be the set of all finite binary sequences of $0$s and $1$s. For any $s\in \{0,1\}^{<\N}$ let $x_s$ and $y_s$ be the self-homeomorphisms of $\{0,1\}^\N$ defined by
\[
\xi \cdot x_s \defeq \left\{\begin{array}{ll}
s(\eta\cdot x) & \text{if }\xi=s\eta\\
\xi & \text{otherwise}
\end{array}\right. \text{ and }
\xi \cdot y_s \defeq \left\{\begin{array}{ll}
s(\eta\cdot y) & \text{if }\xi=s\eta\\
\xi & \text{otherwise.}
\end{array}\right.
\]

Thompson's group $F$ is the subgroup of $\Homeo(\{0,1\}^\N)$ generated by all the $x_s$ for $s\in \{0,1\}^{<\N}$. The Lodha--Moore groups are defined as follows. First, $_yG_y$ is the subgroup of $\Homeo(\{0,1\}^\N)$ generated by Thompson's group $F$ together with all the $y_s$ for $s\in \{0,1\}^{<\N}$. Next, $G_y$ is the subgroup generated by $F$ together with all the $y_s$ such that $s$ is not of the form $0^n$, $_yG$ is the subgroup generated by $F$ together with all the $y_s$ such that $s$ is not of the form $1^n$, and $G$ is the subgroup generated by $F$ together with all the $y_s$ such that $s$ is not of the form $0^n$ or $1^n$. We will focus mostly on $G$ here, and results about $G_y$, $_yG$, and $_yG_y$ will follow easily from results about $G$.

\subsection{Characters}\label{sec:chars}

The BNSR-invariants of a group are certain sets of equivalences classes of characters of the group (we will wait until Subsection~\ref{sec:bnsr} to give the technical definition). Here a \emph{character} of a group is a homomorphism from the group to the additive real numbers. In order to understand the BNSR-invariants of a group $\Gamma$, one first needs to understand the euclidean vector space of characters $\Hom(\Gamma,\R)$.

Let us recall the structure of $\Hom(G,\R)$. It follows easily from the presentation of $G$ given in \cite{lodha16} that the abelianization of $G$ is $\Z^3$ and so $\Hom(G,\R)\cong\R^3$. As a basis for $\Hom(G,\R)$ we will use the characters denoted $\chi_0,\chi_1,\psi$ in \cite{zaremsky16}. These are defined as follows. We use the notation $\mathcal{L}=\{0,1\}^{<\N}\setminus \{0^n,1^n\mid n\ge 0\}$, so $G$ is generated by all $x_s$ for $s\in\{0,1\}^{<\N}$ and all $y_s$ for $s\in\mathcal{L}$.

\[
\chi_0(x_s)=\begin{cases}
-1&\text{ if }s\in \{0^n\mid n\ge 0\}\\
0&\text{ if }s\notin \{0^n\mid n\ge 0\}\\
\end{cases}
\qquad \chi_0(y_s)=0, \forall s\in \mathcal{L},
\]

\[
\chi_1(x_s)=
\begin{cases}
1&\text{ if }s\in \{1^n\mid n\ge 0\}\\
0&\text{ if }s \notin \{1^n\mid n\ge 0\}\\
\end{cases}
\qquad \chi_1(y_s)=0, \forall s\in \mathcal{L}
\]

\[
\psi(x_s)=0, \forall s\in \{0,1\}^{<\N}
\qquad \psi(y_s)=1, \forall s\in \mathcal{L} \text{.}
\]

For the other Lodha--Moore groups $H$ we also have $\Hom(H,\R)\cong \R^3$, but with different bases for different $H$. This is explained in \cite{zaremsky16}, and we recall it here. First define two more characters, $\psi_0$ and $\psi_1$:

\[
\psi_0(x_s)=0, \forall s\in \{0,1\}^{<\N}
\qquad \psi_0(y_s)=\begin{cases}
1&\text{ if }s\in \{0^n\mid n\ge 0\}\\
0&\text{ if }s\notin \{0^n\mid n\ge 0\}\\
\end{cases}
\]

\[
\psi_1(x_s)=0, \forall s\in \{0,1\}^{<\N}
\qquad \psi_1(y_s)=\begin{cases}
1&\text{ if }s\in \{1^n\mid n\ge 0\}\\
0&\text{ if }s\notin \{1^n\mid n\ge 0\} \text{.}
\end{cases}
\]

Now $\Hom(G_y,\R)$ has basis $\{\chi_0,\psi_1,\psi\}$, $\Hom(_yG,\R)$ has basis $\{\psi_0,\chi_1,\psi\}$, and $\Hom(_yG_y,\R)$ has basis $\{\psi_0,\psi_1,\psi\}$. Note that $\psi_0$ is trivial on $G$ and $G_y$, $\psi_1$ is trivial on $G$ and $_yG$, $\chi_0$ is not well defined on $_yG$ or $_yG_y$, and $\chi_1$ is not well defined on $G_y$ or $_yG_y$. In particular for each Lodha--Moore group $H$, precisely three of the characters $\chi_0,\chi_1,\psi_0,\psi_1,\psi$ are well defined and non-trivial on $H$, and these form a basis for $\Hom(H,\R)$.

\section{Cluster complexes}\label{sec:cluster_cpxes}

In this section we recall the notion of cluster complexes, as described by the first author in \cite{lodha14}.

Consider $\R^n$ with variables $x_1,\dots,x_n$, representing coordinates in the usual orthonormal basis. We view $\{x_1,\dots,x_n\}$ as an ordered set, with the order induced by the usual ordering on the indices.

A \emph{hyperplane arrangement} $\mathcal{A}$ is a finite set of affine hyperplanes in $\R^n$. A \emph{region} of $\mathcal{A}$ is a connected component of $\R^n\setminus \bigcup_{H\in \mathcal{A}}H$. Let $\mathcal{R(A)}$ denote the set of regions of $\mathcal{A}$. Note that every region is open and convex, and hence homeomorphic to the interior of the $n$-dimensional ball.

The \emph{flats} of a hyperplane arrangement $\mathcal{A}$ are the affine subspaces of $\R^n$ obtained by taking intersections of some number of hyperplanes in $\mathcal{A}$. Note that the trivial intersection counts as a flat, and equals $\R^n$. Given a flat $T$, we denote by $\mathcal{A}\restriction T$ the hyperplane arrangement consisting of the set of hyperplanes
\[
\{ T\cap H\mid H\in \mathcal{A}, T\not \subseteq H\}
\]
in the subspace $T$. We define the regions of $T$ in a similar way as above, i.e., $\mathcal{R(A)}\restriction T$ equals the set of connected components of $T\setminus  \bigcup_{H\in \mathcal{A}\restriction T} H$.

The union of the regions
\[
\bigcup_{T\text{ is a flat of }\mathcal{A}} \mathcal{R(A)}\restriction T
\]
is called the \emph{face complex} of the arrangement $\mathcal{A}$. The face complex provides a cellular structure on $\R^n\bigcup \partial\R^n$.

We consider two types of hyperplanes in $\R^n$.
\begin{enumerate}
 \item A hyperplane is of \emph{type $1$} if it is of the form $\{x_i=0\}$ or $\{x_i=1\}$ for some $1\le i\le n$.
 \item A hyperplane is of \emph{type $2$} if it is of the form $\{x_i=x_{i+1}\}$ for some $1\le i\le n-1$.
\end{enumerate}

Let $\mathcal{A}$ be a hyperplane arrangement in $\R^n$. We say that $\mathcal{A}$ is \emph{admissible} if:
\begin{enumerate}
 \item Each hyperplane in $\mathcal{A}$ is of type $1$ or of type $2$, and
 \item $\mathcal{A}$ contains all the hyperplanes of type $1$.
\end{enumerate}

\begin{definition}[Cluster]\label{def:cluster}
An \emph{$n$-cluster} is a CW subdivision of $[0,1]^n$ obtained by restricting the face complex of an admissible hyperplane arrangement to $[0,1]^n$. For an admissible hyperplane arrangement $\mathcal{A}$ in $\R^n$, we denote the corresponding cluster by $\mathcal{C(A)}$, and call it the $n$-cluster (or simply cluster) \emph{associated with} $\mathcal{A}$.
\end{definition}

Now we have a bijective correspondence:
\[
\{\text{admissible hyperplane arrangements in }\R^n\}\leftrightarrow \{\text{$n$-clusters}\}\text{.}
\]

\begin{example}
As an easy example, the admissible arrangement given by the set of all type $1$ hyperplanes corresponds to the standard CW structure of the regular euclidean cube $[0,1]^n$. As a non-cubical example, if $\mathcal{A}$ consists of all type $1$ hyperplanes of $\R^2$ and the type $2$ hyperplane $\{x_1=x_2\}$, then the resulting $2$-cluster is a square with a diagonal connecting $(0,0)$ and $(1,1)$. This has four $0$-cells, five $1$-cells, and two triangular $2$-cells.
\end{example}

Our next step is to define the notion of a \emph{subcluster} of a cluster. Fix an admissible hyperplane arrangement $\mathcal{A}$, and let $\mathcal{C(A)}\subseteq \R^n$ be the associated $n$-cluster.

\begin{definition}[Subcluster]
A \emph{subcluster} of $\mathcal{C(A)}$ is a subcomplex obtained by intersecting $\mathcal{C(A)}$ with a flat in $\mathcal{A}$. If the flat is a hyperplane in $\mathcal{A}$, call the subcluster \emph{codimension-$1$}.
\end{definition}

A codimension-$1$ subcluster automatically inherits a description in terms of an admissible hyperplane arrangement in $\R^{n-1}$. To see this, let $H\in \mathcal{A}$ be a hyperplane, so $H$ is of the form
\[
\{x_k=x_{k+1}\}\text{, }\qquad \{x_k=1\}\text{, }\qquad\text{ or }\qquad \{x_k=0\}
\]
for some $1\le k\le n$. In each case we endow $H\cong \R^{n-1}$ with a set of coordinates $y_1,\dots,y_{n-1}$, where $y_j=x_j$ if $j<k$ and $y_j=x_{j+1}$ if $k\le j\le n-1$. The restriction $\mathcal{A}\restriction H$ provides an admissible hyperplane arrangement for $H$. Iterating this, we see that any subcluster inherits a description in terms of an admissible hyperplane arrangement. In particular a subcluster of a cluster is itself a cluster.

We categorize the subclusters of a given cluster as \emph{facial} and \emph{diagonal}. This is done by first separating the flats in $\mathcal{A}$ into two types. We say that a flat in $\mathcal{A}$ is of \emph{type 1} if it is of the form $\bigcap_{H\in A}H$,
where $A\subseteq \mathcal{A}$ is a collection of type $1$ hyperplanes. We say that a flat $T$ in $\mathcal{A}$ is of \emph{type 2} if is of the form $T=\bigcap_{H\in A}H$, for some $A\subseteq \mathcal{A}$ such that $A$ contains at least one (type 2) hyperplane of the form $\{x_k=x_{k+1}\}$ and $T$ is not contained in any of the hyperplanes
\[
\{x_k=0\}\text{, }\qquad \{x_k=1\}\text{, }\qquad \{x_{k+1}=0\}\text{, }\qquad \text{ or }\qquad \{x_{k+1}=1\}\text{.}
\]

\begin{definition}[Facial, diagonal subclusters]
Given a cluster, consider a subcluster obtained by taking an intersection of the cluster with a flat $T$. If $T$ is a type $1$ flat, then the subcluster is a \emph{facial subcluster}. If $T$ is a type $2$ flat, then the subcluster is a \emph{diagonal subcluster}.
\end{definition}

Let us record the following remark, which will be important later in Subsection~\ref{sec:morse}.

\begin{remark}\label{rmk:cluster_cvx}
The cells of a cluster are copies of convex polyhedra in euclidean space. The vertex set of any cell in a cluster is convex independent, i.e., no vertex lies in the convex span of the others, and any cell equals the convex span of its vertices.
\end{remark}

We can now define the notion of a cluster complex.

\begin{definition}[Cluster complex]
A \emph{cluster complex} is a CW complex obtained by gluing clusters along subclusters using cellular homeomorphisms.
\end{definition}

Note that if a cluster complex $C$ is obtained by gluing clusters along facial subclusters (never along diagonal subclusters), then $C$ is a subdivision of a cube complex. In this case let us call $C$ a \emph{cubical cluster complex}, and we will sometimes identify it with the cube complex of which it is a subdivision. If $C'$ is a subcomplex of a cluster complex $C$ such that $C'$ is a cubical cluster complex, we will call it a \emph{cubical subcomplex} of $C$.

\begin{remark}
A cube complex is \emph{non-positively curved} if the link of every vertex is a flag complex, meaning a simplicial complex where every finite collection of vertices pairwise spanning edges spans a simplex. One important property of non-positively curved cube complexes is that they are aspherical. In \cite{lodha14}, an analogous notion of a \emph{non-positively curved cluster complex} was defined. It was shown in \cite{lodha14} that the cluster complex $X$ defined in the next section is non-positively curved in this sense.
\end{remark}

\section{The complex $X$}\label{sec:our_complex}

In this section we recall the description of the complex of clusters $X$ on which $G$ acts from \cite{lodha14}. The $0$-skeleton $X^{(0)}$ is the set of right cosets $Fg$ of Thompson's group $F$ in $G$. To describe the $1$-skeleton we need some definitions.

\begin{definition}[Consecutive, alternating, special form]
We call a pair $s,t$ of elements of $\{0,1\}^{<\N}$ \emph{consecutive} if there exists $u\in \{0,1\}^{<\N}$ and $m,n\in\N\cup\{0\}$ such that $s=u01^m$ and $t=u10^n$. Call a list $s_1,\dots,s_k$ \emph{consecutive} if each pair $s_i,s_{i+1}$ is consecutive. Call a list $\varepsilon_1,\dots,\varepsilon_k$ of elements of $\{\pm1\}$ \emph{alternating} if $\varepsilon_{i+1}=-\varepsilon_i$ for all $i$. A word in the generators $y_s$ is called a \emph{special form} if it is of the form $y_{s_1}^{\varepsilon_1}\cdots y_{s_k}^{\varepsilon_k}$ for $s_1,\dots,s_k$ consecutive and $\varepsilon_1,\dots,\varepsilon_k$ alternating.
\end{definition}

Let $\mathcal{S}$ be the set of all elements of $G$ represented by special forms. We will use the same symbol $\nu$ for a special form and the element of $G$ it represents, and this will not cause any confusion. Note that since we are in $G$, any $y_s$ that we use in a special form must satisfy that $s$ is not of the form $0^n$ or $1^n$.

\begin{definition}[The $1$-skeleton]
$X^{(1)}$ is the simplicial graph defined by declaring that two vertices $Fg$ and $Fg'$ in $X$ are connected by an edge whenever $Fg'=F\nu g$ for some $\nu\in \mathcal{S}$. 
\end{definition}

Note that $G$ is generated by $F$ and $\mathcal{S}$, so $X^{(1)}$ is connected.

\subsection{Higher dimensional cells}\label{sec:higher_cells}

Now we shall describe the higher dimensional cells of $X$. Let us first recall the following.

\begin{definition}[Independent]
Given a collection $s_1,\dots,s_k$ of elements of $\{0,1\}^{<\N}$, the elements are \emph{independent} if no $s_i$ is a prefix of any other $s_j$. Note that if $\nu=y_{s_1}^{\varepsilon_1}\cdots y_{s_k}^{\varepsilon_k}$ is a special form then $s_1,\dots,s_k$ are independent. Two special forms $\nu_1,\nu_2$ are \emph{independent} if $\nu_1=y_{s_1}^{\varepsilon_1}\cdots y_{s_k}^{\varepsilon_k}$ and $\nu_2=y_{t_1}^{\delta_1}\cdots y_{t_\ell}^{\delta_\ell}$ with 
each $s_i$ independent from each $t_j$. More generally special forms $\nu_1,\dots,\nu_m$ are \emph{independent} if they are pairwise independent. Note that independent special forms pairwise commute.
\end{definition}

Given a vertex $Fg$, called the \emph{base}, and independent special forms $\nu_1,\dots,\nu_k$, called the \emph{parameters}, it was shown in \cite{lodha14} that these data describe a natural copy of the $1$-skeleton of a $k$-cluster living in $X^{(1)}$. Namely, the vertices of the $k$-cluster are the $F\nu_{i_1}\cdots\nu_{i_p}g$ for $i_1,\dots,i_p\in\{1,\dots,k\}$ (the order of multiplication does not matter, since the $\nu_i$ are independent), and there is an edge between each $F\nu_{j_1}\cdots\nu_{j_q}g$ and $F\nu_{i_1}\cdots\nu_{i_p} g$ whenever
\[
\nu_{j_1}\cdots \nu_{j_q}\nu^{-1}_{i_1}\cdots \nu^{-1}_{i_p}
\]
is a special form.

Let us illustrate this with some examples.

\begin{example}
The induced subgraph on the vertices $F$, $Fy_{01}$, $Fy_{10}^{-1}$, $Fy_{01}y_{10}^{-1}$ is not a square, but rather a square with one diagonal edge from $F$ to $Fy_{01}y_{10}^{-1}$. This edge is present because $y_{01}y_{10}^{-1}$ is a special form. The other diagonal edge is not present, since $y_{01}y_{10}$ is not a special form. In contrast, the induced subgraph on the vertices $F$, $Fy_{01}$, $Fy_{110}^{-1}$, $Fy_{01}y_{110}^{-1}$ is a square, since neither $y_{01}y_{110}^{-1}$ nor $y_{01}y_{110}$ is a special form (since $01,110$ is not consecutive) and so neither diagonal edge is present.
\end{example}

\begin{example}
If our base is $F$ and our parameters are $y_{s0}$, $y_{s10}^{-1}$, and $y_{s11}$ for some $s\ne 0^n,1^n$ (see Figure~\ref{fig:crowded_cluster}), then in particular the corresponding $3$-cube has the long diagonal edge from $F$ to $Fy_{s0}y_{s10}^{-1}y_{s11}$ because $y_{s0}y_{s10}^{-1}y_{s11}$ is a special form. One can compute that $x_sy_{s0}y_{s10}^{-1}y_{s11}=y_s$, and so $Fy_{s0}y_{s10}^{-1}y_{s11}=Fy_s$, i.e., the long diagonal of this cube has endpoints $F$ and $Fy_s$. This shows that we cannot hope to avoid subdivisions of cubes by only using all the $y_s$ rather than all the special forms.
\end{example}

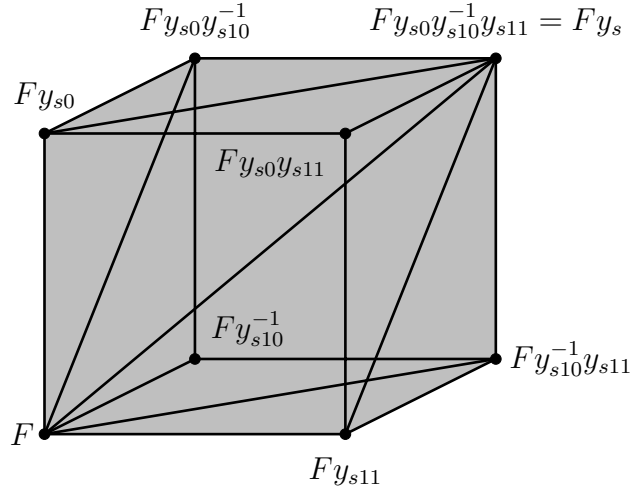
\begin{figure}[htb]
\centering\begin{tikzpicture}

\coordinate (a) at (0,0);
\coordinate (b) at ($(a)+(4,0)$);
\coordinate (c) at ($(a)+(0,4)$);
\coordinate (d) at ($(a)+(4,4)$);
\coordinate (e) at ($(a)+(2,1)$);
\coordinate (f) at ($(b)+(2,1)$);
\coordinate (g) at ($(c)+(2,1)$);
\coordinate (h) at ($(d)+(2,1)$);

\filldraw[lightgray] (a) -- (b) -- (f) -- (h) -- (g) -- (c) -- (a);

\draw[line width=1] (a) -- (b) -- (d) -- (c) -- (a) -- (e) -- (f) -- (h) -- (g) -- (e)   (b) -- (f)   (d) -- (h)   (c) -- (g);
\draw[line width=1] (a) -- (g)   (a) -- (f)   (a) -- (h)   (b) -- (h)   (c) -- (h);

\filldraw (a) circle (2pt);
\filldraw (b) circle (2pt);
\filldraw (c) circle (2pt);
\filldraw (d) circle (2pt);
\filldraw (e) circle (2pt);
\filldraw (f) circle (2pt);
\filldraw (g) circle (2pt);
\filldraw (h) circle (2pt);

\node at ($(a)+(-.3,0)$) {$F$};
\node at ($(b)+(0,-.5)$) {$Fy_{s11}$};
\node at ($(c)+(0,.5)$) {$Fy_{s0}$};
\node at ($(d)+(-1,-.4)$) {$Fy_{s0}y_{s11}$};
\node at ($(e)+(.7,.4)$) {$Fy_{s10}^{-1}$};
\node at ($(f)+(1,0)$) {$Fy_{s10}^{-1}y_{s11}$};
\node at ($(g)+(0,.5)$) {$Fy_{s0}y_{s10}^{-1}$};
\node at ($(h)+(0,.5)$) {$Fy_{s0}y_{s10}^{-1}y_{s11}=Fy_s$};

\end{tikzpicture}
\caption{The $3$-cluster based at $F$ parameterized by the special forms $y_{s0}$, $y_{s10}^{-1}$, and $y_{s11}$.}
\label{fig:crowded_cluster}
\end{figure}

Now the higher dimensional cells of $X$ are simply given by adding clusters to $1$-skeletons of clusters in $X^{(1)}$. It was shown in \cite{lodha14} that this filling is well defined and that the resulting complex $X$ is a complex of clusters. (In particular, the intersection of two clusters is a subcluster of each.) For details we refer the reader to \cite{lodha14}, and we shall make explicit any fact from that section that we use in this article.

See Figures~\ref{fig:crowded_cluster}, \ref{fig:medium_cluster}, and~\ref{fig:empty_cluster} for some examples of $3$-clusters in $X$. More diagonal edges are present when more products of the parameters are themselves special forms. For example $y_{s0}y_{s10}^{-1}y_{s11}$ is a special form, so the long diagonal is present in Figure~\ref{fig:crowded_cluster}, but $y_{s0}y_{s10}^{-1}y_{s111}$ and $y_{s00}y_{s10}^{-1}y_{s111}$ are not (since neither $s0,s10,s111$ nor $s00,s10,s111$ is consecutive), so the long diagonal is not present in Figures~\ref{fig:medium_cluster} or~\ref{fig:empty_cluster}.

\begin{figure}[htb]
\centering\begin{tikzpicture}

\coordinate (a) at (0,0);
\coordinate (b) at ($(a)+(4,0)$);
\coordinate (c) at ($(a)+(0,4)$);
\coordinate (d) at ($(a)+(4,4)$);
\coordinate (e) at ($(a)+(2,1)$);
\coordinate (f) at ($(b)+(2,1)$);
\coordinate (g) at ($(c)+(2,1)$);
\coordinate (h) at ($(d)+(2,1)$);

\filldraw[lightgray] (a) -- (b) -- (f) -- (h) -- (g) -- (c) -- (a);

\draw[line width=1] (a) -- (b) -- (d) -- (c) -- (a) -- (e) -- (f) -- (h) -- (g) -- (e)   (b) -- (f)   (d) -- (h)   (c) -- (g);
\draw[line width=1] (a) -- (g)   (b) -- (h);

\filldraw (a) circle (2pt);
\filldraw (b) circle (2pt);
\filldraw (c) circle (2pt);
\filldraw (d) circle (2pt);
\filldraw (e) circle (2pt);
\filldraw (f) circle (2pt);
\filldraw (g) circle (2pt);
\filldraw (h) circle (2pt);

\node at ($(a)+(-.3,0)$) {$F$};
\node at ($(b)+(0,-.5)$) {$Fy_{s111}$};
\node at ($(c)+(0,.5)$) {$Fy_{s0}$};
\node at ($(d)+(-1,-.4)$) {$Fy_{s0}y_{s111}$};
\node at ($(e)+(.7,.4)$) {$Fy_{s10}^{-1}$};
\node at ($(f)+(1,0)$) {$Fy_{s10}^{-1}y_{s111}$};
\node at ($(g)+(0,.5)$) {$Fy_{s0}y_{s10}^{-1}$};
\node at ($(h)+(0,.5)$) {$Fy_{s0}y_{s10}^{-1}y_{s111}$};

\end{tikzpicture}
\caption{The $3$-cluster based at $F$ parameterized by the special forms $y_{s0}$, $y_{s10}^{-1}$, and $y_{s111}$.}
\label{fig:medium_cluster}
\end{figure}
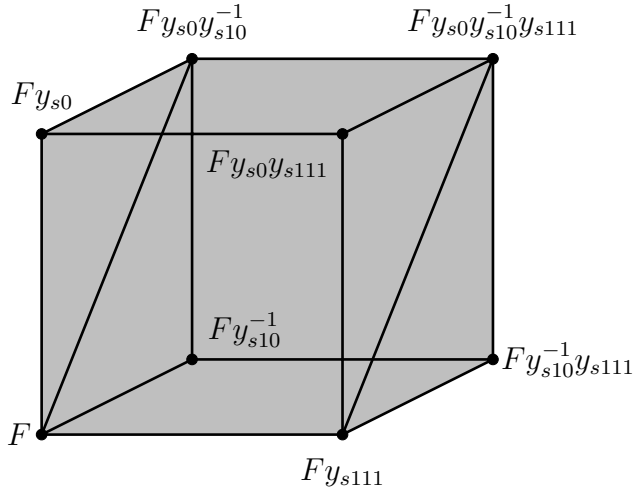

\begin{figure}[htb]
\centering\begin{tikzpicture}

\coordinate (a) at (0,0);
\coordinate (b) at ($(a)+(4,0)$);
\coordinate (c) at ($(a)+(0,4)$);
\coordinate (d) at ($(a)+(4,4)$);
\coordinate (e) at ($(a)+(2,1)$);
\coordinate (f) at ($(b)+(2,1)$);
\coordinate (g) at ($(c)+(2,1)$);
\coordinate (h) at ($(d)+(2,1)$);

\filldraw[lightgray] (a) -- (b) -- (f) -- (h) -- (g) -- (c) -- (a);

\draw[line width=1] (a) -- (b) -- (d) -- (c) -- (a) -- (e) -- (f) -- (h) -- (g) -- (e)   (b) -- (f)   (d) -- (h)   (c) -- (g);

\filldraw (a) circle (2pt);
\filldraw (b) circle (2pt);
\filldraw (c) circle (2pt);
\filldraw (d) circle (2pt);
\filldraw (e) circle (2pt);
\filldraw (f) circle (2pt);
\filldraw (g) circle (2pt);
\filldraw (h) circle (2pt);

\node at ($(a)+(-.3,0)$) {$F$};
\node at ($(b)+(0,-.5)$) {$Fy_{s111}$};
\node at ($(c)+(0,.5)$) {$Fy_{s00}$};
\node at ($(d)+(-1,-.4)$) {$Fy_{s00}y_{s111}$};
\node at ($(e)+(.7,.4)$) {$Fy_{s10}^{-1}$};
\node at ($(f)+(1,0)$) {$Fy_{s10}^{-1}y_{s111}$};
\node at ($(g)+(0,.5)$) {$Fy_{s00}y_{s10}^{-1}$};
\node at ($(h)+(0,.5)$) {$Fy_{s00}y_{s10}^{-1}y_{s111}$};

\end{tikzpicture}
\caption{The $3$-cluster based at $F$ parameterized by the special forms $y_{s00}$, $y_{s10}^{-1}$, and $y_{s111}$.}
\label{fig:empty_cluster}
\end{figure}
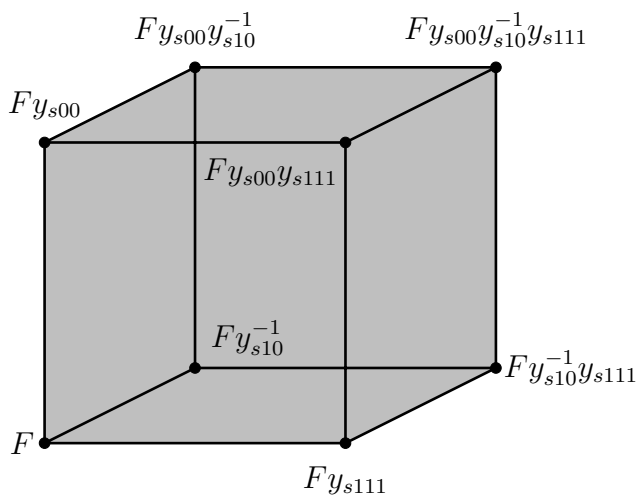

The following summarizes some important properties of $X$, which were all shown in \cite{lodha14}.

\begin{theorem}\cite{lodha14}\label{thrm:mainlodha14}
The action of the group $G$ on the cluster complex $X$ by cell permuting homeomorphisms satisfies the following:
\begin{enumerate}
\item $X$ is contractible.
\item The quotient $X/G$ has finitely many cells in each dimension.
\item The stabilizer of each cell is of type $\F_\infty$.
\end{enumerate}
It follows that the group $G$ is of type $\F_\infty$.
\end{theorem}

\subsection{Links}\label{sec:links}

It will be very important in what follows to understand vertex links in $X$. The \emph{link} $\lk(v)$ of a vertex $v$ in a complex is the space of directions out of $v$. Since the cells of a cluster complex are copies of convex polyhedra (Remark~\ref{rmk:cluster_cvx}), the link $\lk(v)$ also has the structure of a CW complex whose cells are convex polyhedra. Note that since the group $G$ acts transitively on $X^{(0)}$, the links of all vertices are naturally homeomorphic to each other. Hence it suffices to study the link of the vertex represented by the trivial coset $F$.

\begin{lemma}\label{lem:cone}
Let $C_1,\dots,C_k$ be clusters in $X$ based at $F$. Then there exists $m\in\N$ such that for each $1\le i\le k$ there is a cluster $D_i$ based at $F$, containing $C_i$ and the $1$-cluster with vertices $F$ and $Fy_{0^m1}$. In particular every finite subcomplex of $\lk(F)$ lies in a contractible subcomplex of $\lk(F)$, and so $\lk(F)$ is contractible.
\end{lemma}

\begin{proof}
The finitely many $C_i$ use finitely many parameters $\nu_i$, which use finitely many generators $y_s$, which use finitely many subscripts $s$. No such $s$ are of the form $0^n$ or $1^n$, so there exists $m$ such that $0^m1$ is independent from every such $s$. Now $y_{0^m1}$ is independent from every $\nu_i$, and the existence of the $D_i$ is clear. To see that this implies the result about $\lk(F)$, first note that any finite subcomplex of $\lk(F)$ lies in the link of $F$ in a union of finitely many clusters $C_1,\dots,C_k$ based at $F$. Let $U_C=C_1\cup\cdots\cup C_k$, so we are assuming that our finite subcomplex lies in $\lk_{U_C}(F)$. By the first result we can choose clusters $D_1,\dots,D_k$ based at $F$, say with union $U_D=D_1\cup\cdots\cup D_k$, such that $\lk_{U_D}(F)$ contains a cone with base $\lk_{U_C}(F)$ and cone point the vertex of $\lk(F)$ represented by $Fy_{0^m1}$. Since this cone is contractible, we are done.
\end{proof}

\section{BNSR-invariants and Morse theory}\label{sec:bnsr_morse}

In this section we discuss the relevant setup on BNSR-invariants and Bestvina--Brady discrete Morse theory.

\subsection{BNSR-invariants}\label{sec:bnsr}

The Bieri--Neumann--Strebel--Renz (BNSR) invariants $\Sigma^n(\Gamma)$ ($n\in\N$) of a type $\F_\infty$ group $\Gamma$ are a family of invariants that catalog the finiteness properties of subgroups of $\Gamma$ containing the commutator subgroup $[\Gamma,\Gamma]$. 
The invariant $\Sigma^1$ was introduced by Bieri, Neumann, and Strebel in \cite{bieri87}, and the higher invariants $\Sigma^n$ were introduced by Bieri and Renz in \cite{bieri88}.

The definition of $\Sigma^n(\Gamma)$ is a bit involved. First, consider the euclidean space $\Hom(\Gamma,\R)$ of characters $\chi$ of $\Gamma$, and call two characters \emph{equivalent} if they are positive scalar multiples of each other. The equivalence classes $[\chi]$ of non-trivial characters, called \emph{character classes}, form the \emph{character sphere} $\Sigma(\Gamma)$. If $\Hom(\Gamma,\R)$ is $d$ dimensional then $\Sigma(\Gamma)$ is homeomorphic to a $(d-1)$-sphere. The invariants $\Sigma^n(\Gamma)$ are certain subsets of $\Sigma(\Gamma)$, satisfying
\[
\Sigma(\Gamma)\supseteq \Sigma^1(\Gamma)\supseteq \Sigma^2(\Gamma)\supseteq\cdots \text{.}
\]
Before defining the $\Sigma^n(\Gamma)$, let us point out their main application:

\begin{theorem}\cite[Theorem~1.1]{bieri10}\label{thrm:bnsr_fin_props}
For $\Gamma$ a group of type $\F_n$ and $\Gamma'$ a subgroup of $\Gamma$ containing the commutator subgroup $[\Gamma,\Gamma]$, we have that $\Gamma'$ is of type $\F_n$ if and only if $[\chi]\in\Sigma^n(\Gamma)$ for all $\chi$ satisfying $\chi(\Gamma')=0$.
\end{theorem}

For example if $\Gamma'$ is the kernel of a character $\chi$ whose image is a copy of $\Z$, called a \emph{discrete character}, then $\Gamma'$ is of type $\F_n$ if and only if $[\pm\chi]\in\Sigma^n(\Gamma)$.
Now we state the technical definition of the BNSR-invariants. (There are multiple equivalent ways to define them; for our purposes this will be the most useful, coming from \cite[Definition~1.1]{zaremsky17PB_n}.)

\begin{definition}[BNSR-invariants]\label{def:bnsr}
Let $\Gamma$ be a group of type $\F_n$, and let $\chi\in\Hom(\Gamma,\R)$ be a character. Let $Y$ be an $(n-1)$-connected CW complex on which $\Gamma$ acts cocompactly and with cell stabilizers contained in $\ker(\chi)$. There exists a map $h_\chi\colon Y\to\R$ satisfying
\[
h_\chi(g.y)=\chi(g)+h_{\chi}(y)
\]
for all $g\in \Gamma$ and $y\in Y$. For $t\in\R$ let $Y_{\chi\ge t}$ be the full subcomplex of $Y$ supported on those vertices $y$ with $\chi(y)\ge t$. The $n$th \emph{Bieri--Neumann--Strebel--Renz (BNSR) invariant} $\Sigma^n(\Gamma)$ is the subset of $\Sigma(\Gamma)$ consisting of all those $[\chi]$ for which the filtration $(Y_{\chi\ge t})_{t\in\R}$ is essentially $(n-1)$-connected, meaning that for all $t$ there exists $s\le t$ such that the inclusion $Y_{\chi\ge t}\to Y_{\chi\ge s}$ induces the trivial map in all homotopy groups $\pi_k$ for $k\le n-1$.
\end{definition}

As a remark, such a $Y$ and $h_\chi$ always exist and the definition is independent of the choices of $Y$ and $h_\chi$ (see \cite[Remark~8.2]{bux04} and \cite[Theorem~12.1]{bieri03}). This defines $\Sigma^n(\Gamma)$, and for $\Gamma$ of type $\F_\infty$ we define $\Sigma^\infty(\Gamma)$ to be the intersection of all the $\Sigma^n(\Gamma)$.

Let us now give a sufficient condition for a character class to lie in $\Sigma^\infty(\Gamma)$, which will be useful for our purposes since it does not require the action of $\Gamma$ on $Y$ to be cocompact, just cocompact on skeleta.

\begin{lemma}\label{lem:bnsr_non_cocpt}
Let $\Gamma$ be a group of type $\F_\infty$. Let $Y$ be a contractible CW complex on which $\Gamma$ acts with finitely many orbits of cells in each dimension and with cell stabilizers lying in $\ker(\chi)$. Let $h_\chi$ and $Y_{\chi\ge t}$ be as in Definition~\ref{def:bnsr}. If for all $s\le t$ the inclusion $Y_{\chi\ge t}\to Y_{\chi\ge s}$ is a homotopy equivalence, then $[\chi]\in\Sigma^\infty(\Gamma)$.
\end{lemma}

\begin{proof}
We need to show that $[\chi]\in\Sigma^n(\Gamma)$ for arbitrary $n\in\N$. Since $Y$ is contractible its $n$-skeleton is $(n-1)$-connected, and the action of $\Gamma$ on $Y^{(n)}$ is cocompact, so by Definition~\ref{def:bnsr} $[\chi]\in\Sigma^n(\Gamma)$ if and only if the filtration $(Y_{\chi\ge t}^{(n)})_{t\in\R}$ is essentially $(n-1)$-connected. Since a cellular map between CW complexes induces the trivial map in $\pi_k$ if and only if this is true of the restricted map between the complexes' $(k+1)$-skeleta, we see that $[\chi]\in\Sigma^n(\Gamma)$ if and only if the filtration $(Y_{\chi\ge t})_{t\in\R}$ is essentially $(n-1)$-connected. Since $Y$ is contractible, the inclusion $Y_{\chi\ge t}\to Y$ induces the trivial map in $\pi_k$ for all $k$. Since every element of $\pi_k(Y_{\chi\ge t})$ maps to $0$ in $\pi_k(Y)$ and since homotopy spheres are compact, every such element must already map to $0$ in $\pi_k(Y_{\chi\ge s})$ for some $s\le t$, which since the inclusions $Y_{\chi\ge t}\to Y_{\chi\ge s}$ are all homotopy equivalences implies that in fact $\pi_k(Y_{\chi\ge t})=0$ for all $k$ and $t$. In particular the filtration $(Y_{\chi\ge t})_{t\in\R}$ is essentially $(n-1)$-connected (for trivial reasons).
\end{proof}

\subsection{Morse theory}\label{sec:morse}

Bestvina--Brady discrete Morse theory, introduced in \cite{bestvina97}, is a useful topological tool that reduces difficult problems to easier ones. For example in Lemma~\ref{lem:bnsr_non_cocpt} we would like the inclusion $Y_{\chi\ge t}\to Y_{\chi\ge s}$ to be a homotopy equivalence, which is a difficult ``global'' problem, but we can use Morse theory to reduce this to the easier ``local'' problem of checking that certain subcomplexes of links are contractible.

Let us begin to set up our Morse theoretic framework. The type of complex we work with is an affine cell complex, that is a (regular) CW complex whose cells are copies of convex polyhedra in euclidean space, glued together along faces. For example this includes simplicial complexes, cube complexes, and (thanks to Remark~\ref{rmk:cluster_cvx}) cluster complexes. Affine cell complexes come with the usual notion of the \emph{link} of a vertex, which is the space of directions out of the vertex (as has already been defined for cluster complexes).

A nice feature of affine cell complexes is that, once we understand the topology of the link of a vertex, we can understand the topological relationship between such a complex and the full subcomplex supported on all vertices but that one, as we now explain.

\begin{lemma}[Add one vertex]\label{lem:one_vtx}
Let $Y$ be an affine cell complex. Let $y$ be a vertex of $Y$, and let $Z$ be the full subcomplex of $Y$ supported on all vertices of $Y$ except $y$. If the link of $y$ in $Y$ is $(n-1)$-connected then the inclusion $Z\to Y$ induces an isomorphism in $\pi_k$ for all $k\le n-1$ and a surjection in $\pi_n$.
\end{lemma}

\begin{proof}
Since $Y$ is an affine cell complex there is a strong deformation retraction from $Y\setminus\{y\}$ to $Z$, given by, on each cell containing $y$, radially projecting away from $y$ to the boundary of the cell. Hence it suffices to show that the inclusion $Y\setminus\{y\} \to Y$ induces such maps. Note that $Y$ is the union of $Y\setminus\{y\}$ with a contractible neighborhood of $y$, and the intersection of these factors is homotopy equivalent to the link of $y$ in $Y$. The result now follows from standard topological tools, e.g., Seifert--van Kampen, Mayer--Vietoris, and Hurewicz.
\end{proof}

Now we define the notion of a Morse function on an affine cell complex. The original definition from \cite{bestvina97} is not general enough for our purposes, but numerous generalizations have been made in the literature. For us it will be most convenient to use a variation of the notion of ``Morse function'' from \cite[Definition~2.1]{zaremsky17PB_n}. Our definition here is slightly different, which we will spell out afterward.

\begin{definition}[Morse function]\label{def:morse_function}
Let $Y$ be an affine cell complex and let $h,f\colon Y^{(0)}\to \R$ be a pair of real valued functions on the vertex set of $Y$. Consider the function $(h,f)\colon Y^{(0)}\to \R\times \R$ with the codomain ordered via the usual lexicographic order. We call $(h,f)$ a \emph{Morse function} provided the following conditions hold:
\begin{enumerate}
\item Every cell has a unique vertex at which $(h,f)$ achieves a minimum value.
\item There exists $\varepsilon>0$ such that for any pair of adjacent vertices $y$ and $z$, either $|h(y)-h(z)|\ge\varepsilon$ or $h(y)=h(z)$.
\item The set $f(Y^{(0)})\subseteq \R$ is inverse well ordered, i.e., every non-empty subset of $f(Y^{(0)})$ has a maximal element.
\end{enumerate}
\end{definition}

The main difference between this definition and \cite[Definition~2.1]{zaremsky17PB_n} is that there, $h$ and $f$ were defined on all of $Y$ and required to be affine on cells and non-constant on edges. This is a sufficient condition for every cell to have a unique vertex at which $(h,f)$ achieves a minimum value, but it is not necessary. Also, we should remark that in \cite[Definition~2.1]{zaremsky17PB_n} this was called an ``ascending-type Morse function'', and a ``descending'' version involved requiring a well order instead of an inverse well order. We will only need the ascending version here, so we just call this a ``Morse function''.

The ``certain subcomplexes of links'' referenced above as being important are the following:

\begin{definition}[Ascending link]
Let $(h,f)$ be a Morse function on an affine cell complex $Y$. The \emph{ascending star} $\asst y$ of a vertex $y$ is the subcomplex of $Y$ consisting of all cells on which $(h,f)$ achieves its minimum at $y$, and their faces. The \emph{ascending link} $\alk y$ of $y$ is the link of $y$ in $\asst y$.
\end{definition}

Next we state the all-important Morse Lemma, which is Lemma~2.3 in \cite{zaremsky17PB_n}. Since our definition of Morse function has technically not appeared in the literature before, let us re-prove the Morse Lemma using this definition, but the proof is essentially identical to that of \cite[Lemma~2.3]{zaremsky17PB_n}.

\begin{lemma}[Morse Lemma]\label{lem:morse}
Let $(h,f)$ be a Morse function on an affine cell complex $Y$. For $t\in\R\cup\{-\infty\}$ let $Y_{h\ge t}$ be the full subcomplex of $Y$ supported on those vertices $y$ with $h_\chi(y)\ge t$. Let $-\infty\le s<t$. If for all vertices $y$ with $s\le h(y)<t$ the ascending link $\alk y$ is $(n-1)$-connected, then the inclusion $Y_{h\ge t}\to Y_{h\ge s}$ induces an isomorphism in $\pi_k$ for $k\le n-1$ and a surjection in $\pi_n$.
\end{lemma}

\begin{proof}
First, without loss of generality we assume that $t-s\le \varepsilon$, thanks to induction and compactness of homotopy spheres. Let $Y_{s\le h<t}$ be the full subcomplex of $Y$ supported on vertices $y$ with $s\le h(y)<t$, so we can build up from $Y_{h\ge t}$ to $Y_{h\ge s}$ by attaching, in some order, the vertices of $Y_{s\le h<t}$ along relative links. We need to choose the order to make the topology well behaved. To this end, pick the inverse well order $\prec$ on the vertex set of $Y_{s\le h<t}$ satisfying that $y\prec z$ whenever $f(y)<f(z)$. This is possible thanks to $f(Y^{(0)})\subseteq \R$ being inverse well ordered. The purpose of using this inverse well ordering is that, since every subset of ``new'' vertices has a maximum element under this order, there is always a well defined ``next'' vertex to add at every stage, allowing for transfinite induction.

Now we claim that the relative link of a vertex in $Y_{s\le h<t}$ with respect to this order equals its ascending link with respect to the Morse function $(h,f)$. The relative link of $y$ is the full subcomplex of the link of $y$ in $Y$ supported on those vertices $z$ with either $h(z)\ge t$, or else $s\le h(z)<t$ and $y\prec z$. (Here if $z$ is a vertex of $Y$ adjacent to $y$, then we are also writing $z$ to denote the corresponding vertex of the link of $y$ in $Y$.) To see that this equals $\alk y$, first note that for any vertex $z$ adjacent to $y$, either $|h(y)-h(z)|\ge\varepsilon$ or $h(y)=h(z)$. If $|h(y)-h(z)|\ge\varepsilon$, then since $t-s\le \varepsilon$ we know $s\le h(z)<t$ does not hold, so $z$ is in the relative link of $y$ if and only if $h(z)\ge t$ if and only if $z$ is in the ascending link of $y$. If $h(y)=h(z)$ (so $f(y)\ne f(z)$) then $z$ is in the relative link of $y$ if and only if $y\prec z$, and $z$ is in the ascending link of $y$ if and only if $f(y)<f(z)$. But we constructed $\prec$ so that these are equivalent.

Now when we build up from $Y_{h\ge t}$ to $Y_{h\ge s}$ by attaching the vertices of $Y_{s\le h<t}$ according to the order $\prec$, at each stage the relative link is the ascending link, hence is $(n-1)$-connected by hypothesis. Thanks to Lemma~\ref{lem:one_vtx}, this implies that attaching any given vertex along its relative link induces an isomorphism in $\pi_k$ for $k\le n-1$ and a surjection in $\pi_n$. Since $\prec$ is an inverse well order, transfinite induction says that the same is true after attaching every missing vertex, which is to say the inclusion $Y_{h\ge t}\to Y_{h\ge s}$ induces an isomorphism in $\pi_k$ for $k\le n-1$ and a surjection in $\pi_n$.
\end{proof}

\begin{corollary}\label{cor:morse}
With the same setup as the Morse Lemma~\ref{lem:morse}, if for all vertices $y$ with $s\le h(y)<t$ the ascending link $\alk y$ is contractible, then the inclusion $Y_{h\ge t}\to Y_{h\ge s}$ is a homotopy equivalence.
\end{corollary}

\begin{proof}
By the Morse Lemma~\ref{lem:morse} the inclusion induces an isomorphism in $\pi_k$ for all $k$, and so it is a homotopy equivalence by Whitehead's Theorem.
\end{proof}

\section{The BNSR-invariants of the Lodha--Moore groups}\label{sec:main}

In this section we prove our main results, Theorem~\ref{thrm:main} and Corollary~\ref{cor:main_cor}, by applying Morse theory to the cluster complex $X$.

First let us recall the partial computation of $\Sigma^n(G)$ from \cite{zaremsky16}. We have $\Sigma^1(G)=\Sigma(G)\setminus\{[\chi_0],[\chi_1]\}$, and $\Sigma^2(G)$ equals $\Sigma^1(G)$ with the convex hull of $[\chi_0]$ and $[\chi_1]$ removed, i.e., we remove any character class of the form $[a\chi_0+b\chi_1]$ for $a,b\ge 0$. From Lemma~5.8 of \cite{zaremsky16} one can see that as soon as $[\pm\psi]\in\Sigma^\infty(G)$ we will have that $\Sigma^n(G)=\Sigma^2(G)$ for all $n\ge 2$. Combining this with \cite[Observation~5.6]{zaremsky16} gives $\Sigma^n(H)=\Sigma^2(H)$ for all $n\ge 2$ for all Lodha--Moore groups $H$.

Thus our main task is to prove that $[\pm\psi]\in\Sigma^\infty(G)$.

The homomorphism $\psi\colon G\to\Z$ induces a function $h_\psi\colon X^{(0)}\to\Z$ via $h_\psi(Fg)\defeq\psi(g)$. This is well defined since $\psi(F)=0$ and $\psi$ is a homomorphism. Now let $f\colon X^{(0)}\to -\N$ be an arbitrary injection from $X^{(0)}$ to the set of negative integers $-\N$ (since $G$ is finitely generated it is countable, and hence so is $X^{(0)}=F\backslash G$).

\begin{lemma}\label{lem:psi_morse}
Viewed as a function from $X^{(0)}$ to $\R\times\R$, $(h_\psi,f)$ is a Morse function (in the sense of Definition~\ref{def:morse_function}).
\end{lemma}

\begin{proof}
Since $f$ is injective, distinct vertices have distinct $(h_\psi,f)$ values, and so every cell has a unique vertex with minimum $(h,f)$ value. Since the image of $h_\psi$ is $\Z$, the condition involving $\varepsilon$ holds using $\varepsilon=1$. Since the image of $f$ is $-\N$, it is inverse well ordered.
\end{proof}

\begin{theorem}\label{thrm:main}
For any Lodha--Moore group $H$ and any $n\ge 2$ we have $\Sigma^n(H)=\Sigma^2(H)$.
\end{theorem}

\begin{proof}
By the results from \cite{zaremsky16} recalled above, it suffices to show that $[\pm\psi]\in\Sigma^\infty(G)$. We will show that $[\psi]\in\Sigma^\infty(G)$ and a parallel argument shows $[-\psi]\in\Sigma^\infty(G)$. We will first verify that $G$, $X$, and $\psi$ satisfy the hypotheses of Lemma~\ref{lem:bnsr_non_cocpt}. Recall from Theorem~\ref{thrm:mainlodha14} that $X$ is contractible, and that $X/G$ has finitely many cells in each dimension. Next, we wish to show that stabilizers in $G$ of cells in $X$ lie in the kernel of $\psi$. First note that every vertex stabilizer is a conjugate of $F$, and hence lies in $\ker(\psi)$. Also, any cell stabilizer contains as a finite index subgroup the intersection of the stabilizers of its vertices. This shows that any cell stabilizer has a finite index subgroup that lies in $\ker(\psi)$, and hence the stabilizer itself must lie in $\ker(\psi)$. This verifies the hypotheses of Lemma~\ref{lem:bnsr_non_cocpt}.

Now it suffices to show that for any $s\le t$ the inclusion $X_{\psi\ge t}\to X_{\psi\ge s}$ is a homotopy equivalence. By Lemma~\ref{lem:psi_morse}, $(h_\psi,f)\colon X^{(0)}\to \R\times\R$ is a Morse function, so by Corollary~\ref{cor:morse} it suffices to show that every ascending link of a vertex of $X$ with respect to $(h_\psi,f)$ is contractible. Thanks to the action of $G$, we only need to consider the vertex represented by the trivial coset $F$. Given any finite subcomplex of $\alk(F)$, by Lemma~\ref{lem:cone} the subcomplex lies in a contractible cone (in $\lk(F)$) on a vertex of the form $Fy_{0^m1}$. Moreover, since $\psi(y_{0^m1})=1$ we know $Fy_{0^m1}$ lies in $\alk(F)$, and so the subcomplex lies in the cone in $\alk(F)$ on $Fy_{0^m1}$. Since spheres are compact, this shows that every homotopy sphere in $\alk(F)$ is nullhomotopic, so by Whitehead's Theorem $\alk(F)$ is contractible. (The parallel argument for $-\psi$ uses vertices of the form $Fy_{0^m1}^{-1}$.)
\end{proof}

We quickly obtain:

\begin{corollary}\label{cor:main_cor}
Every finitely presented normal subgroup of $G$ is of type $\F_\infty$.
\end{corollary}

\begin{proof}
Since every non-trivial normal subgroup of $G$ contains the commutator subgroup $[G,G]$ \cite{burillo18}, this is immediate from Theorem~\ref{thrm:bnsr_fin_props} and Theorem~\ref{thrm:main}.
\end{proof}

\subsection{Finiteness properties of normal subgroups of $G$}\label{sec:classify_normal}

Since the non-trivial normal subgroups of $G$ are in natural one-to-one correspondence with the subgroups of $\Z^3$, we can use Theorem~\ref{thrm:bnsr_fin_props} and the full computation of the BNSR-invariants of $G$ to get a complete picture of the finiteness properties of normal subgroups. Let $\alpha\colon G\to \Z^3$ be the abelianization map $\alpha=(\chi_0,\chi_1,\psi)$, and for $i=1,2,3$ let $\pi_i\colon \Z^3\to \Z$ be projection onto the $i$th coordinate. The non-trivial normal subgroups of $G$ are precisely of the form $\alpha^{-1}(A)$ for $A\le \Z^3$. The classification is as follows:

\begin{corollary}\label{cor:classify}
If $\pi_1(A)=\{0\}$ or $\pi_2(A)=\{0\}$ then $\alpha^{-1}(A)$ is not finitely generated. If this is not the case but $(a\pi_1+b\pi_2)(A)=\{0\}$ for some $a,b>0$ then $\alpha^{-1}(A)$ is finitely generated but not finitely presented. In all other cases $\alpha^{-1}(A)$ is of type $\F_\infty$.
\end{corollary}

\begin{proof}
In all of this, we use Theorem~\ref{thrm:bnsr_fin_props}. If $\pi_1(A)=\{0\}$ then $\alpha^{-1}(A)\le \ker(\chi_0)$, and since $[\chi_0]\not\in \Sigma^1(G)$, $\alpha^{-1}(A)$ is not finitely generated. A similar argument handles the $\pi_2(A)=\{0\}$ case. Now assume $\pi_1(A)\ne\{0\}$ and $\pi_2(A)\ne\{0\}$, so $\alpha^{-1}(A)$ does not lie in the kernel of $\chi_0$ or $\chi_1$. Since $\Sigma^1(G)=\Sigma(G)\setminus\{[\chi_0],[\chi_1]\}$ this implies $\alpha^{-1}(A)$ is finitely generated. If $(a\pi_1+b\pi_2)(A)=\{0\}$ for some $a,b>0$ then $\alpha^{-1}(A)$ lies in the kernel of $a\chi_0+b\chi_1$, which since $[a\chi_0+b\chi_1]\not\in\Sigma^2(G)$ implies $\alpha^{-1}(A)$ is not finitely presented. Now assume this is not the case, so $\alpha^{-1}(A)$ does not lie in the kernel of $a\chi_0+b\chi_1$ for any $a,b\ge 0$. Since $\Sigma^\infty(G)=\Sigma(G)\setminus \{[a\chi_0+b\chi_1]\mid a,b\ge 0\}$ we get that $\alpha^{-1}(A)$ is of type $\F_\infty$.
\end{proof}

To give some easy examples, if $\Z^3=\Z e_1 \oplus \Z e_2 \oplus \Z e_3$ then $\alpha^{-1}(\Z e_1)$ is not finitely generated, $\alpha^{-1}(\Z (e_1-e_2))$ is finitely generated but not finitely presented, and $\alpha^{-1}(\Z (e_1+e_2))$ is of type $\F_\infty$.

Finally let us recall the computations of $\Sigma^1(H)$ and $\Sigma^2(H)$ from \cite{zaremsky16} for the other Lodha--Moore groups $H$, for the record, now that we know $\Sigma^\infty(H)=\Sigma^2(H)$. First we have $\Sigma^1(G_y)=\Sigma(G_y)\setminus\{[\chi_0],[-\psi_1]\}$, and $\Sigma^2(G_y)$ equals $\Sigma^1(G_y)$ with the convex hull of $[\chi_0]$ and $[-\psi_1]$ removed. Next we have $\Sigma^1(_yG)=\Sigma(_yG)\setminus\{[\psi_0],[\chi_1]\}$, and $\Sigma^2(_yG)$ equals $\Sigma^1(_yG)$ with the convex hull of $[\psi_0]$ and $[\chi_1]$ removed. Finally we have $\Sigma^1(_yG_y)=\Sigma(_yG_y)\setminus\{[\psi_0],[-\psi_1]\}$, and $\Sigma^2(_yG_y)$ equals $\Sigma^1(_yG_y)$ with the convex hull of $[\psi_0]$ and $[-\psi_1]$ removed. From this one can easily classify which normal subgroups of a Lodha--Moore group $H$ containing the commutator subgroup $[H,H]$ are not finitely generated, which are finitely generated but not finitely presented, and which are of type $\F_\infty$, and we leave this as an exercise for the reader.

\section{Application: An exotic type $\F_\infty$ simple group}\label{sec:circle}

In \cite{lodha19} the first author constructed a finitely presented infinite simple group $S$ of homeomorphisms of the circle that does not admit any non-trivial action by $C^1$-diffeomorphisms, nor by piecewise linear homeomorphisms, on any $1$-manifold. The construction is closely related to that of $_yG_y$, and we can now use our results to prove that $S$ is also of type $\F_\infty$. In particular $S$ is a member of the very interesting and important class of simple groups of type $\F_\infty$.

Let us recall the definition of $S$ from \cite{lodha19}. First we need to define Thompson's group $T$ as a group of homeomorphisms of $\{0,1\}^\N$. To this end, we define a family of homeomorphisms $p_n\colon \{0,1\}^\N \to \{0,1\}^\N$ for $n\ge 0$ as follows:

\[
\xi \cdot p_n \defeq \left\{\begin{array}{ll}
1^{k+1}0\eta & \text{if }\xi=1^k 0\eta \text{ for } 0\le k\le n-1\\
1^{n+1}\eta & \text{if }\xi=1^n 0\eta\\
0\eta & \text{if }\xi=1^{n+1}\eta \text{.}
\end{array}\right.
\]

Note that when $n=0$ the topmost case does not occur, and $p_0$ is simply the map
\[
\xi \cdot p_0 \defeq \left\{\begin{array}{ll}
1\eta & \text{if }\xi=0\eta\\
0\eta & \text{if }\xi=1\eta
\end{array}\right. \text{.}
\]

Thompson's group $T$ is the group of homeomorphisms of $\{0,1\}^\N$ generated by Thompson's group $F$ and all the $p_n$. Standard facts about $T$ are that it is simple and of type $\F_\infty$. Ghys and Sergiescu also proved in \cite{ghys87} that $T$ admits a faithful action on the circle by $C^\infty$-diffeomorphisms.

Now we can define the group $S$ from \cite{lodha19}.

\begin{definition}\label{def:S}
The group $S$ is the group of homeomorphisms of $\{0,1\}^\N$ generated by Thompson's group $T$ together with the element $y_{10}y_{110}^{-1}$.
\end{definition}

\begin{remark}
The definition of $S$ above is different from the definition given in the introduction. However the prescribed action of $S$ from the introduction, as a group of homeomorphisms of $\mathbb{S}^1=\R\cup\{\infty\}$, is semiconjugate to the above action. The semiconjugacy is provided by the map $\Phi\colon \{0,1\}^\N\to \R\cup \{\infty\}$ defined below in Subsection~\ref{sec:HatSFinfty}.
\end{remark}

By \cite[Theorem~1.1]{lodha19}, $S$ is a finitely presented infinite simple group that admits a faithful action by homeomorphisms on the circle, but admits no non-trivial action by $C^1$-diffeomorphisms, nor by piecewise linear homeomorphisms, on any $1$-manifold. The main goal of this section is to prove that $S$ is of type $\F_\infty$:

\begin{theorem}\label{thrm:circle}
The group $S$ is of type $\F_\infty$.
\end{theorem}

As a remark, to prove this one could hope to mimic the proof from \cite{lodha14} that $G$ is of type $\F_\infty$, using cluster complexes, but adapting this to $S$ turns out to pose significant challenges.

\subsection{The group $\widehat{S}$}

An important ingredient in the proof of Theorem~\ref{thrm:circle} will be the group of homeomorphisms of $\{0,1\}^\N$ generated by $_yG_y$ and $T$, which we denote by $\widehat{S}$. Thus $\widehat{S}$ is generated by the families $x_s$, $y_s$, and $p_n$, for all $s\in \{0,1\}^{<\N}$ and $n\ge 0$. This group is also interesting for its own sake; it is the ``circle version'' of the Lodha-Moore groups in an analogous way to how $T$ is the ``circle version'' of Thompson's group $F$. In this section we shall produce an explicit infinite presentation for $\widehat{S}$ and compute its abelianization. It will be a consequence of our work that the group $S$ is precisely the kernel of this abelianization map.

\begin{lemma}
The group $\widehat{S}$ satisfies:
\[
\widehat{S} = \langle _yG_y,T\rangle = \langle G_y,T\rangle = \langle _yG,T\rangle = \langle G,T\rangle = \langle y_{10},T\rangle \text{.}
\]
\end{lemma}

\begin{proof}
We only need to show that $_yG_y\le \langle y_{10},T\rangle$, for which it suffices to prove that $y_s\in \langle y_{10},T\rangle$ for all $s\in \{0,1\}^{<\N}$. Recall that $T$ admits a transitive partial action on the set $\{0,1\}^{<\N}\setminus\{\emptyset\}$, and so $y_s\in \langle y_{10},T\rangle$ for all $\emptyset\ne s\in \{0,1\}^{<\N}$. Moreover, thanks to the relation $y_\emptyset=x_\emptyset y_0y_{10}^{-1}y_{11}$ we know that $y_\emptyset\in \langle y_{10},T\rangle$ as well.
\end{proof}

We now provide an explicit infinite presentation for the group $\widehat{S}$. The following is the infinite generating set $\mathcal{X}$:

\begin{enumerate}
\item ($y$-generators) $y_{s}$ for $s\in \{0,1\}^{<\N}$
\item ($x$-generators) $x_{s}$ for $s\in \{0,1\}^{<\N}$
\item ($p$-generators) $p_n$ for $n\ge 0$.
\end{enumerate}

We now list a set of relations $\mathcal{R}$ in the generating set $\mathcal{X}$. We separate them in three families. In what follows below, we have $s,t\in \{0,1\}^{<\N}$ and $n\ge 0$.

First, the relations in the $x$-generators.
\begin{enumerate}
\item[(1)] $x_s^2=x_{s0}x_sx_{s1}$
\item[(2)] $x_{s}x_t=x_tx_{s\cdot x_t}$ if $s\cdot x_t$ is defined.
\end{enumerate}

Secondly, the relations in the $x$-generators and $p$-generators.
\begin{enumerate}
\item[(3)] $x_{1^m}^{-1} p_n x_{1^{m+1}}=p_{n+1}\text{ if }n<m\qquad p_n x_\emptyset=p_{n+1}^2$.
$p_n=x_{1^n}p_{n+1}\qquad p_n^{n+2}=1$.
\end{enumerate}

Finally, the relations involving the $y$-generators:
\begin{enumerate}
\item[(4)] $y_s=x_sy_{s0}y_{s10}^{-1}y_{s11}$.
\item[(5)] $y_sx_t=x_ty_{s\cdot x_t}$ if $s\cdot x_t$ is defined.
\item[(6)] $y_sp_n=p_ny_{s\cdot p_n}$ if $s\cdot p_n$ is defined.
\item[(7)] $y_sy_t=y_ty_s$ if $s,t$ are independent.
\end{enumerate}

We remark that the $x$- and $y$-generators with the relations $(1)-(2)$ and $(4)-(7)$ provide an infinite presentation for the group $_yG_y$. If we impose the suitable restrictions for $s,t$ in the occurrences of $y_s,y_t$ in the presentation, then this provides an infinite presentation for the groups $_yG,G_y,G$. This was the main result of \cite{lodha16}. Moreover, the $x$-generators with the relations $(1)-(2)$ provide a presentation for $F$, and the $x,p$-generators with the relations $(1)-(3)$ provide a presentation for $T$.

To prove that $\langle \mathcal{X}\mid \mathcal{R}\rangle$ is a presentation for $\widehat{S}$ we need the following notion of a standard form.

\begin{definition}[Standard form]
Define a partial order $<$ on $\{0,1\}^{<\N}$ as follows. We have $s_i<s_j$ if either of the following holds:
\begin{enumerate}
\item There is a sequence $t\in \{0,1\}^{<\N}$ such that $s_i=s_jt$.
\item There are sequences $t_1,t_2,t_3\in \{0,1\}^{<\N}$ such that $s_i=t_10t_2$ and $s_j=t_11t_3$.
\end{enumerate}
A word of the form
\[
fy_{s_1}^{t_1}\dots y_{s_n}^{t_n}
\]
is said to be a \emph{standard form} if $f\in T$ and $s_i<s_j$ whenever $i<j$.
\end{definition}

\begin{lemma}\label{lem:standardform}
Every word in the generating set $\mathcal{X}$ can be converted using relations in $\mathcal{R}$ into a word in standard form.
\end{lemma}

\begin{proof}
The proof is virtually identical to the proof of the analogous Lemma~5.4 in \cite{lodha16} concerning the groups $_yG_y$ and $G$. It is an elementary exercise to adapt the proof to this setup with the additional relations in $(3)$.
\end{proof}

\begin{proposition}\label{prop:presentation}
$\widehat{S}=\langle \mathcal{X}\mid \mathcal{R}\rangle$.
\end{proposition}

\begin{proof}
Given a word $g=w_1\dots w_n$ in the generating set $\mathcal{X}$ that equals the identity in $\widehat{S}$, we need to show that we can use the relations to reduce it to the identity. First, we reduce the word to one in standard form $g=fy_{s_1}^{t_1}\dots y_{s_m}^{t_m}$ using Lemma~\ref{lem:standardform}, where $f$ is a word in the generators of $T$. Note that any word of the form $y_{s_1}^{t_1}\dots y_{s_m}^{t_m}$ must fix $0^\infty$ and $1^\infty$. Since $g$ equals the identity, it follows that $f$ also fixes $0^\infty$ and $1^\infty$. Recall that the subgroup of $T$ consisting of elements that fix $0^\infty$ and $1^\infty$ is precisely Thompson's group $F$. Since $f$ is a word in the $x$- and $p$-generators, and since $f\in F$, we can use the relations $(1)-(3)$ to reduce it to a word in the $x$-generators of $F$. We now have a word in the generating set of $_yG_y$ which equals the identity. In \cite{lodha16} it was shown that this can be reduced to the empty word using the relations $(1)-(2)$ and $(4)-(7)$.
This completes the proof.
\end{proof}

\begin{lemma}\label{lem:abelianization}
The abelianization $\widehat{S}/[\widehat{S},\widehat{S}]$ equals $\Z$. Moreover, the image of $y_s$ for any $s\in \{0,1\}^{<\N}$ in the abelianization equals $1$.
\end{lemma}

\begin{proof}
First note that since $T$ is simple, it lies in the kernel of the abelianization. Next, since the partial action of $T$ on the set of non-empty finite binary sequences is transitive, using relations $(5),(6)$ we conclude that $y_s$ and $y_t$, $s,t\in \{0,1\}^\N\setminus\{\emptyset\}$ are conjugate by elements of $T$. This implies that their images are equal. From relation (4) we also get that the image of $y_\emptyset$ equals those of all the other $y_s$. It follows from ``abelianizing'' the relations that this image is nontrivial and generates a copy of $\Z$.
\end{proof}

Recall the homomorphism $\psi\colon _yG_y\to \Z$ defined in Subsection~\ref{sec:chars}.

\begin{lemma}\label{lem:psi_is_ablnztn}
The homomorphism $\psi$ extends to a homomorphism $\widehat{\psi}\colon \widehat{S}\to \Z$ with $T\le \ker(\widehat{\psi})$. Moreover, this coincides with the abelianization map $\widehat{S}\to \widehat{S}/[\widehat{S},\widehat{S}]=\Z$.
\end{lemma}

\begin{proof}
Recall that $F\le \ker(\psi)$. Since $T$ is simple, it follows that the normal closure of $\ker(\psi)$ in $\widehat{S}$ contains $T$ as a subgroup. It follows immediately from the definition of the relations $\mathcal{R}$ and Lemma~\ref{lem:abelianization} that $\psi$ extends to a map $\widehat{\psi}$ on $\widehat{S}$ and this extension corresponds to the aforementioned abelianization map.
\end{proof}

\subsection{The group $S$}

Now we return our focus on the group $S$ from Definition~\ref{def:S}, and establish some facts about it. In particular we will prove that $S=\ker(\widehat{\psi})$ (Proposition~\ref{prop:kernel}).

\begin{lemma}\label{lem:pairs}
The element $y_sy_t^{-1}$ lies in $S$ for each independent pair $s,t\in \{0,1\}^{<\N}$.
\end{lemma}

\begin{proof}
We consider two cases.

{\bf Case $1$}: First assume $s,t$ are consecutive and satisfy that $\{s,t\}\ne \{0,1\}$. The partial action of $T$ on consecutive pairs $\{s,t\}$ satisfying $\{s,t\}\ne \{0,1\}$ is transitive. Hence there exists an element $f\in T$ such that $f^{-1} y_{10}y_{110}^{-1}f = y_{10\cdot f}y_{110\cdot f}^{-1}=y_sy_t^{-1}$.

{\bf Case $2$} Next assume $s,t$ are not consecutive. Note that $y_{10}y_{1110}^{-1}\in S$ since
\[
y_{10}y_{110}^{-1} x_\emptyset^{-1} y_{10}y_{110}^{-1} x_\emptyset = y_{10}y_{1110}^{-1} \text{.}
\]
The partial action of $T$ on non-consecutive pairs of sequences is transitive. Hence there exists an element $f\in T$ such that $f^{-1} y_{10}y_{1110}^{-1}f=y_{10\cdot f}y_{1110\cdot f}^{-1}=y_sy_t^{-1}$.

The remaining case is when $\{s,t\}=\{0,1\}$, and we leave it as a pleasant exercise for the reader to verify that $y_0y_1^{-1}\in S$. This uses the relation $y_1=x_1y_{10}y_{110}^{-1}y_{111}$.
\end{proof}

\begin{lemma}\label{lem:balancedword0}
A word of the form $y_{1^{m+1}0}^{s_1}y_{1^{m+2}0}^{s_2}\dots y_{1^{m+n}0}^{s_n}$ represents an element of $S$ whenever $\sum_{1\le i\le n}s_i=0$.
\end{lemma}

\begin{proof}
From Lemma~\ref{lem:pairs} it follows that any word of the form $y_{1^{m+i}0}^{\pm 1}y_{1^{m+j}0}^{\mp 1}$ lies in $S$ for each pair $i,j\in \{1,\dots,n\}$. It is straightforward to check that any word satisfying our hypothesis is a product of such words.
\end{proof}

\begin{lemma}\label{lem:balancedword1}
Let $y_{s_1}^{t_1}\dots y_{s_n}^{t_n}$ be a word satisfying:
\begin{enumerate}
\item $\sum_{1\leq i\le n}t_i=0$
\item $s_i\notin \{1^m\mid m\ge 0\}$ for each $1\le i\le n$.
\end{enumerate}
Then this word represents an element of $S$.
\end{lemma}

\begin{proof}
There is an $m\in \N$ such that $\{s_1,\dots,s_n\}$ lies in the complement of the infinite subtree rooted at $1^m$. We obtain
\begin{align*}
(y_{s_1}^{t_1}\dots y_{s_n}^{t_n}) (y_{s_n}^{-t_n} y_{1^{m+1}0}^{t_n}) (y_{s_{n-1}}^{-t_{n-1}}y_{1^{m+2}0}^{t_{n-1}})\dots (y_{s_1}^{-t_1}y_{1^{m+n}0}^{t_1})\\
=y_{1^{m+1}0}^{-t_n}\dots y_{1^{m+n}0}^{-t_1} \text{.}
\end{align*}
This is an element of $S$ thanks to Lemma \ref{lem:balancedword0}. Since each $y_{s_i}^{-t_i}y_{1^{m+n+1-i}0}^{t_i}$ is in $S$ by Lemma~\ref{lem:pairs}, we conclude that $y_{s_1}^{t_1}\dots y_{s_n}^{t_n}$ is in $S$.
\end{proof}

\begin{proposition}\label{prop:kernel}
We have $S=\ker(\widehat{\psi})$.
\end{proposition}

\begin{proof}
It is clear that $S\le \ker(\widehat{\psi})$ from the definitions. Recall that any element in $\ker(\widehat{\psi})$ can be represented by a standard form
\[
fy_{s_1}^{t_1}\dots y_{s_n}^{t_n}\qquad f\in T, \sum_{1\le i\le n} t_i=0
\]
using the relations. Since we know that $f\in T \le S$, it suffices to show that $y_{s_1}^{t_1}\dots y_{s_n}^{t_n}\in S$. 

Since $y_{s_1}^{t_1}\dots y_{s_n}^{t_n}$ is a standard form, by definition it must look like:
\[
(y_{p_1}^{q_1}\dots y_{p_m}^{q_m} ) (y_{1^{k_1}}^{j_1}\dots y_{1^{k_l}}^{j_l})\qquad p_1,\dots,p_n\notin \{1^k\mid k\ge 0\} \text{.}
\]
Upon multiplying by finitely many words of the form described in Lemma~\ref{lem:pairs}, we obtain
\[
(y_{p_1}^{q_1}\dots y_{p_m}^{q_m} ) (y_{1^{k_1}}^{j_1}\dots y_{1^{k_l}}^{j_l}) (y_{1^{k_l}}^{-j_l}y_{0}^{j_l})\dots (y_{1^{k_1}}^{-j_1}y_{0}^{j_1}) \text{.}
\]
Upon applying a sequence of relations of the form $y_sy_t=y_ty_s$ (where $s,t$ are independent) and $y_sy_s^{-1}=1$ we obtain a $y$-word that satisfies the hypothesis of Lemma~\ref{lem:balancedword1}, and hence lies in $S$. It follows that our original word must lie in $S$.
\end{proof}

We will use the fact that $S=\ker(\widehat{\psi})$ to prove Theorem~\ref{thrm:circle}. Our strategy is to first prove that $\widehat{S}$ is of type $\F_\infty$, and then prove that $[\pm\widehat{\psi}]\in\Sigma^\infty(\widehat{S})$, which by Theorem~\ref{thrm:bnsr_fin_props} will imply that $S=\ker(\widehat{\psi})$ is of type $\F_\infty$. As a remark, by Lemma~\ref{lem:psi_is_ablnztn} we have $S=[\widehat{S},\widehat{S}]$, and so our work will in fact show that $\Sigma(\widehat{S})=\Sigma^\infty(\widehat{S})$.

\subsection{The proof that $\widehat{S}$ is of type $\F_\infty$}\label{sec:HatSFinfty}

In this subsection we prove that $\widehat{S}$ is of type $\F_\infty$. Let
\[
Z\defeq \{s0^\infty\mid s\in \{0,1\}^{<\N}\}
\]
be the set of all points in $\{0,1\}^\N$ that are eventually constant $0$s. Since every $x_s$, $y_s$, and $p_n$ stabilizes $Z$, the group $\widehat{S}$ stabilizes $Z$. Let $\Delta(Z)$ be the simplex on $Z$, i.e., the contractible simplicial complex with vertex set $Z$ and a $k$-simplex for each $(k+1)$-element subset $\{s_0 0^\infty,\dots,s_k 0^\infty\}\subseteq Z$. At this point we have a contractible simplicial complex $\Delta(Z)$ on which the group $\widehat{S}$ acts, and our goal is to understand this action.

To study this action we shall need the following alternative description of $\widehat{S}$. We identify the circle $\mathbb{S}^1$ with the real projective line $\R\cup \{\infty\}$. Recall the following map $\Phi$ from \cite{lodha16}. (We will state some basic facts about this map and refer the reader to this reference for more details.)
\[
\Phi\colon \{0,1\}^\N\to \R\cup \{\infty\}
\]
\[
\Phi(00^{n_1}1^{n_2}0^{n_3}\dots)\mapsto -(n_1+\frac{1}{n_2+\frac{1}{n_3+\frac{1}{n_4+\dots}}})\qquad \Phi(11^{n_1}0^{n_2}1^{n_3}\dots)\mapsto n_1+\frac{1}{n_2+\frac{1}{n_3+\frac{1}{n_4+\dots}}}
\]
The map $\Phi$ satisfies that the inverse image of each irrational is unique, the inverse image of each rational is a pair of the form $\{s01^\infty,s10^\infty\}$, and the inverse image of $\infty$ is $\{0^\infty,1^\infty\}$.

The map $\Phi$ conjugates the group $T\le \Homeo(\{0,1\}^\N)$ to the group of piecewise $\PSL_2(\Z)$-projective homeomorphisms of $\mathbb{S}^1$ with breakpoints in $\Q$ (finitely many allowable, for each element). The elements $y_s$ are conjugated to piecewise $\PSL_2(\Z[\frac{1}{\sqrt{2}}])$-projective homeomorphisms. This circle action will be a somewhat more natural setting to formulate our argument below. We remark that in this description of the group it is immediate that $\PSL_2(\Z)$ is naturally a subgroup of $T$, and so the action of $T$ on $\Q\cup \{\infty\}$ is transitive. The stabilizer of $\infty$ is isomorphic to Thompson's group $F$. Indeed, this action is topologically conjugate to the piecewise linear action of $F$.

In this model, we obtain the following equivalent description of $\Delta(Z)$. The restriction of $\Phi$ to $Z$ induces a bijection onto $\Q\cup \{\infty\}$. For simplicity of notation, we also denote by $Z$ the set $\Q\cup \{\infty\}$. Again, $\Delta(Z)$ is the simplex on $Z$, i.e., the contractible simplicial complex with vertex set $Z$ and a $k$-simplex for each $(k+1)$-subset of $Z$.

It will be convenient to write the elements of finite subsets of $Z$ in a circular ordering. An ordered $n$-tuple $(r_1,\dots,r_n)$ of distinct points $r_1,\dots,r_n\in \R$ is said to be \emph{linearly ordered} if $r_1<\dots<r_n$. An ordered $n$-tuple $(r_1,\dots,r_n)$ of distinct points $r_1,\dots,r_n\in \mathbb{S}^1=\R\cup \{\infty\}$ is said to be \emph{circularly ordered} if the following holds:
\begin{enumerate}
\item If $\infty\notin \{r_1,\dots,r_n\}$ then $r_1<\dots<r_n$.
\item If $r_i=\infty$, then
\[
r_{i+1}<r_{i+2}<\dots<r_n<r_1<r_2\dots<r_{i-1} \text{.}
\]
\end{enumerate}

\begin{lemma}\label{lem:cocpt_filt}
The piecewise projective action of $T$ on $\mathbb{S}^1=\R\cup \{\infty\}$ is transitive on the set of circularly ordered $n$-tuples $\{(r_1,\dots,r_n)\mid r_1,\dots,r_n\in \Q\cup\{\infty\}\}$ for each $n\in \mathbb{N}$. Hence the action of $\widehat{S}$ on $\Delta(Z)$ has one orbit of simplices in each dimension.
\end{lemma}

\begin{proof}
This will follow from the following two claims:
\begin{enumerate}
\item The action of $T$ on $\Q\cup \{\infty\}$ is transitive.
\item The action of $F$ on $\Q$ is transitive on linearly ordered $n$-tuples.
\end{enumerate}
The first claim is true since $\PSL_2(\Z)$ acts transitively on $\Q\cup \{\infty\}$ and $\PSL_2(\Z)\le T$.

To see that the second claim is true, we recall the following. There is a homeomorphism $\Xi\colon (0,1)\to \R$ that conjugates the piecewise linear action of $F$ on $(0,1)$ to the piecewise projective action of $F$ on $\R$. This homeomorphism restricts to a bijection between the sets $\Z[\frac{1}{2}]\cap (0,1)$ and $\Q$. For an elegant and detailed proof of this we refer the reader to the notes of Jose Burillo on the Lodha-Moore groups; see \cite{Burillo}. Now it is a classical fact that the piecewise linear action of $F$ on linearly ordered $n$-tuples in $\Z[\frac{1}{2}]\cap (0,1)$ has the required transitivity property, and so the second claim follows.
\end{proof}

Our next goal is to understand stabilizers in $\widehat{S}$ of simplices in $\Delta(Z)$.

\begin{lemma}\label{lem:stabs1}
Let $q\in \Q\cup \{\infty\}$. The stabilizer of $q$ in $\widehat{S}$ is isomorphic to $_yG_y$.
\end{lemma}

\begin{proof}
Note that since $T\le \widehat{S}$, and since the action of $T$ on $\Q\cup \{\infty\}$ is transitive, it follows that the stabilizer of $q$ in $\widehat{S}$ is isomorphic to the stabilizer of $\infty$ in $\widehat{S}$. Thus it suffices to show that $\Stab_{\widehat{S}}(\infty)=~_yG_y$.

It is clear that $_yG_y$ is a subgroup of $\Stab_{\widehat{S}}(\infty)$. Now let $g\in \Stab_{\widehat{S}}(\infty)$. Recall from Lemma~\ref{lem:standardform} that we can write $g$ in the form $g=fy_{s_1}^{t_1}\dots y_{s_n}^{t_n}$ such that $f\in T$ is a word in the $x$- and $p$-generators. Since $\infty\cdot y_{s_1}^{t_1}\dots y_{s_n}^{t_n}=\infty$, it follows that $\infty\cdot f=\infty$.
Since the stabilizer of $\infty$ in $T$ is precisely the group $F$, it follows that $f\in F$, and so can be reduced to a word in the $x$-generators using the relations in $T$. Since any word of the form $f y_{s_1}^{t_1}\dots y_{s_n}^{t_n}$ for $f\in F$ represents an element in $_yG_y$, we are done.
\end{proof}

Recall that, given a group action of a group $\Gamma$ by homeomorphisms on the circle and an open set $I\subset \mathbb{S}^1$, the \emph{rigid stabilizer} of $I$ in the group is the subgroup consisting of elements that pointwise fix the complement of $I$. This is denoted as $\RStab_\Gamma(I)$.

\begin{lemma}\label{lem:stabs2}
There is an isomorphism $\xi\colon \RStab_{\widehat{S}}([0,1])\to~_yG_y$ such that
\[
\psi\circ \xi = \widehat{\psi}\restriction \textup{RStab}_{\widehat{S}}([0,1]) \text{.}
\]
\end{lemma}

\begin{proof}
Recall from the above that the map $\Phi\colon \{0,1\}^\N\to \R\cup \{\infty\}$ conjugates the given action of $\widehat{S}$ on $\{0,1\}^\N$ to the action on $\mathbb{S}^1$. Let $U=\{10\eta\mid \eta\in \{0,1\}^\N\}$. The map $\Phi$ naturally restricts to a map $\Phi\colon U\to [0,1]$. This restriction induces a conjugacy between the rigid stabilizer of $U$ in the first action to the rigid stabilizer of $[0,1]$ in the second action. Since each element of the rigid stabilizer of $U$ in $\widehat{S}$ fixes $1^\infty,0^\infty$, this coincides with the rigid stabilizer of $U$ in $_yG_y$, which is naturally isomorphic to $_yG_y$, being the group generated by
\[
\{x_{10s},y_{10t}\mid s,t\in \{0,1\}^{<\N}\} \text{.}
\]
Now we can build an isomorphism $\xi\colon \RStab_{\widehat{S}}([0,1])\to~_yG_y$ by sending each $x_{10s}$ to $x_s$ and each $y_{10s}$ to $y_s$. Clearly
\[
\psi\circ \xi = \widehat{\psi}\restriction \textup{RStab}_{\widehat{S}}([0,1]) \text{.}
\]
\end{proof}

\begin{lemma}\label{lem:stabs}
The stabilizer in $\widehat{S}$ of any simplex in $\Delta(Z)$ is of type $\F_\infty$.
\end{lemma}

\begin{proof}
Since being of type $\F_\infty$ is invariant under passing to finite index, it suffices to prove that the pointwise stabilizer in $\widehat{S}$ of any simplex is of type $\F_\infty$. For this we will show that the pointwise stabilizer of a $k$-simplex is a direct product of $k+1$ copies of $_yG_y$.

Let $\{r_0,\dots,r_k\}\subset\Q\cup \{\infty\}$ be a $k$-simplex. We assume without loss of generality that $r_0,\dots,r_k$ appear in the natural circular ordering. We would like to show that $\RStab_{\widehat{S}}([r_i,r_{i+1}])$ is isomorphic to $_yG_y$ for each $0\le i\le k$ (modulo $k+1$). We know from Lemma~\ref{lem:cocpt_filt} that the action of $T$ on circularly ordered $2$-tuples is transitive. Therefore, there is an element $f\in T$ such that $[r_i,r_{i+1}]\cdot f=[0,1]$. The restriction of $f$ to these intervals induces a natural isomorphism between their rigid stabilizers. Finally, it follows from Lemma~\ref{lem:stabs2} that these groups are isomorphic to $_yG_y$, as desired.
\end{proof}

Now we quickly obtain:

\begin{corollary}
The group $\widehat{S}$ is of type $\F_\infty$.
\end{corollary}

\begin{proof}
The action of $\widehat{S}$ on $\Delta(Z)$ has finitely many orbits of cells in each dimension (Lemma~\ref{lem:cocpt_filt}), and stabilizers of type $\F_\infty$ (Lemma~\ref{lem:stabs}), so by Brown's Criterion \cite[Proposition~3.1 and Corollary~3.3(a)]{brown87} $\widehat{S}$ is of type $\F_\infty$.
\end{proof}

\subsection{The proof that $S$ is of type $\F_\infty$}

Now we want to show that $[\pm\widehat{\psi}]\in\Sigma^\infty(\widehat{S})$, which since $S=\ker(\widehat{\psi})$ will show that $S$ is of type $\F_\infty$. To do this, we need some tools from BNSR-invariant theory, which we record here.

\begin{theorem}\cite[Theorem~3.2]{meier98}\label{thrm:bnsr_stabs}
Let $\Gamma$ be a group acting on a contractible complex with finitely many orbits of cells in each dimension. Let $\chi\in\Hom(\Gamma,\R)$ be a character of $\Gamma$ whose restriction $\chi\restriction \Stab_\Gamma(e)$ to the stabilizer of any cell $e$ is non-trivial. Suppose that $[\chi\restriction \Stab_\Gamma(e)]\in \Sigma^\infty(\Stab_\Gamma(e))$ for every cell $e$. Then $[\chi]\in\Sigma^\infty(\Gamma)$.
\end{theorem}

\begin{theorem}\cite[Theorem~1.2]{bieri10prod}\label{thrm:dir_prod}
Let $A$ and $B$ be groups and let $\Gamma=A\times B$. Let $\chi_A\colon A\to \R$ and $\chi_B\colon B\to \R$ be characters and let $\chi_A+\chi_B \colon \Gamma\to \R$ be the character $(\chi_A+\chi_B)(a,b)=\chi_A(a)+\chi_B(b)$. If $[\chi_A]\in\Sigma^\infty(A)$ and $[\chi_B]\in\Sigma^\infty(B)$ then $[\chi_A+\chi_B]\in \Sigma^\infty(\Gamma)$.
\end{theorem}

This next fact is immediate from Definition~\ref{def:bnsr}, since if an action by $\Gamma$ is cocompact then so is the induced action of any finite index subgroup.

\begin{theorem}\label{thrm:fin_index}
Let $\Gamma$ be a group of type $\F_\infty$ and $\Gamma'\le \Gamma$ a finite index subgroup. Let $\chi\colon \Gamma\to\R$ be a character, and $\chi\restriction \Gamma'$ its restriction to $\Gamma'$. Then $[\chi]\in\Sigma^\infty(\Gamma)$ if and only if $[\chi\restriction\Gamma']\in\Sigma^\infty(\Gamma')$.
\end{theorem}

Now we can prove that $S$ is of type $\F_\infty$:

\begin{proof}[Proof of Theorem~\ref{thrm:circle}]
It suffices to show that $[\pm\widehat{\psi}]\in\Sigma^\infty(\widehat{S})$. Consider the action of $\widehat{S}$ on the contractible complex $\Delta(Z)$. This has finitely many orbits of simplices in each dimension by Lemma~\ref{lem:cocpt_filt}. By Theorem~\ref{thrm:bnsr_stabs} it now suffices to prove that $[\pm\widehat{\psi}\restriction\Stab_{\widehat{S}}(e)]\in \Sigma^\infty(\Stab_{\widehat{S}}(e))$ for every simplex $e$ of $\Delta(Z)$. By the proof of Lemma~\ref{lem:stabs} together with Lemma~\ref{lem:stabs2}, any $\Stab_{\widehat{S}}(e)$ contains a finite index subgroup $D$ isomorphic to a direct product of finitely many copies of $_yG_y$, and the restriction of $\pm\widehat{\psi}$ to $D$ coincides under this isomorphism with the sum of the characters $\pm\psi\colon _yG_y\to\R$ on each copy of $_yG_y$. Since $[\pm\psi]\in\Sigma^\infty(_yG_y)$ by Theorem~\ref{thrm:main}, we conclude from Theorem~\ref{thrm:dir_prod} that $[\pm\widehat{\psi}\restriction D]\in\Sigma^\infty(D)$, which by Theorem~\ref{thrm:fin_index} implies that $[\pm\psi\restriction\Stab_{\widehat{S}}(e)]\in \Sigma^\infty(\Stab_{\widehat{S}}(e))$, as desired.
\end{proof}

Whereas $G$, $G_y$, $_yG$, and $_yG_y$ are ``$F$-like'' and $\widehat{S}$ and $S$ are ``$T$-like'', one could also consider a ``$V$-like'' Lodha--Moore-esque group, e.g., $V(G)\defeq \langle V,G\rangle \le \Homeo(\{0,1\}^\N)$. This is obviously finitely generated. Also, its abelianization is $\Z$, and its commutator subgroup $[V(G),V(G)]$ is simple and coincides with $\langle V,S\rangle \le \Homeo(\{0,1\}^\N)$. These facts can be demonstrated by working out the presentation of $V(G)$ in a manner that is similar to the above for $\widehat{S}$. It is unclear however how to prove that $V(G)$ and $[V(G),V(G)]$ are of type $\F_\infty$, though we suspect this is the case. The methods used here for $\widehat{S}$ and $S$ do not work, and the cluster complex approach to $G$ from \cite{lodha14} appears to be far too complicated. Thus, we leave this as a conjecture:

\begin{conjecture}
The ``$V$-like Lodha--Moore group'' $V(G)$ and its commutator subgroup are of type $\F_\infty$.
\end{conjecture}

\bibliographystyle{alpha}

\end{document}